\documentclass[12pt, reqno, a4paper]{amsart}
\usepackage{txfonts}	
\usepackage{amsmath, amssymb, amsthm, amscd}
\usepackage[T1]{fontenc}
\usepackage{eucal}
\usepackage[dvips]{color}
\usepackage{multicol}
\usepackage[all]{xy}		
\usepackage{graphicx}
\usepackage{color}
\usepackage{colordvi}
\usepackage{xspace}
\usepackage{tikz}
\usepackage{enumerate}
\usepackage{ulem} 
\usepackage[colorlinks,final,backref=page,hyperindex]{hyperref}
\usepackage{bm}
\usepackage{cancel}
\allowdisplaybreaks % page break for formulars

\newcommand {\emptycomment}[1]{} %to remove paragraphs

\newcommand{\delete}[1]{}
\newif\ifflabel\flabelfalse
\ifflabel
\else
	
\fi
\newtheorem{theorem}{Theorem}[section]
\newtheorem{lemma}[theorem]{Lemma}
\newtheorem{corollary}[theorem]{Corollary}
\newtheorem{proposition}[theorem]{Proposition}
\theoremstyle{definition}
\newtheorem{definition}[theorem]{Definition}
\newtheorem{example}[theorem]{Example}

\newcommand{\End}{\mathrm{End}}
\newcommand{\id}{\mathrm{id}}

\newcommand{\ad}{\mathrm{ad}}

\begin{document}

\title{Quasi-triangular and factorizable Poisson bialgebras}

% authors information
\author{Yuanchang Lin}
\address{School of Mathematics, North University of China, Taiyuan 030051, China}
\email{linyuanchang@mail.nankai.edu.cn}

\author{Dilei Lu}
\address{School of Science, Beijing Information Science and Technology University, Beijing 100192, China}
\email{ludyray@bistu.edu.cn}



\date{\today}

\begin{abstract}
In this paper, we introduce the notions of quasi-triangular and factorizable Poisson bialgebras. 
A factorizable Poisson bialgebra induces a factorization of the underlying Poisson algebra. 
We prove that the Drinfeld classical double of a Poisson bialgebra naturally admits a factorizable Poisson bialgebra structure. 
Furthermore, we introduce the notion of quadratic Rota-Baxter Poisson algebras and show that a quadratic Rota-Baxter Poisson algebra of zero weight induces a triangular  Poisson bialgebra. 
Moreover, we establish a one-to-one correspondence between factorizable Poisson bialgebras and quadratic Rota-Baxter Poisson algebras of nonzero weights. 
Finally, we establish the quasi-triangular and factorizable theories for differential antisymmetric infinitesimal (ASI)    bialgebras, and construct quasi-triangular and factorizable Poisson bialgebras from quasi-triangular and factorizable (commutative and cocommutative) differential ASI bialgebras respectively.
\end{abstract}

\maketitle

\tableofcontents

%%%%%%%%%%%%%%%%%%%%%%%%%%%%%%%%%%%%%%%%%%%%%%%%%%%%%%%%%%%%%%%%%%%%%%%%%%%%%%%%
%%%%%%%%%%%%%%%%%%%%%%%%%%%%%%%%%%%%%%%%%%%%%%%%%%%%%%%%%%%%%%%%%%%%%%%%%%%%%%%%
%%%%%%%%%%%%%%%%%%%%%%%%%%%%%%%%%%%%%%%%%%%%%%%%%%%%%%%%%%%%%%%%%%%%%%%%%%%%%%%%
\section{Introduction}
Poisson algebras serve as fundamental structures in various areas of mathematics and mathematical physics, including Poisson geometry \cite{vaisman2012lectures, weinstein1977lectures}, classical and quantum mechanics \cite{arnol2013mathematical, dirac2013lectures, odzijewicz2011hamiltonian}, algebraic geometry \cite{ginzburg2004poisson, polishchuk1997algebraic}, quantization theory \cite{huebschmann1990poisson, kontsevich2003deformation} and quantum groups \cite{chari1995guide, drinfeld1986quantum}.
A Poisson algebra is both a Lie algebra and a commutative associative algebra which are compatible in a certain sense.
\begin{definition}
	{\rm (\cite{lichnerowicz1977varietes, weinstein1977lectures})}
	A \textbf{Poisson algebra} is a triple $(A, [\ , \ ], \cdot)$, 
	where $(A, [\ ,\ ])$ is a Lie algebra and $(A, \cdot)$ is a commutative associative algebra satisfying the following equation:
	\begin{equation*}
		[a, b \cdot c] = [a, b] \cdot c + b \cdot [a, c], \;\;
		\forall a, b, c \in A.
	\end{equation*}
\end{definition}

Both Lie algebras and associative algebras possess well-developed theories of bialgebras, which have found extensive applications in various mathematical fields. 
In the Lie algebra setting, Lie bialgebras, introduced by Drinfeld \cite{drinfeld1983hamiltonian}, play a fundamental role as the infinitesimalisation of Poisson Lie groups \cite{chari1995guide, kassel2012quantum}.
On the side of associative algebras, the notion of infinitesimal bialgebras was introduced by Joni and Rota to provide an algebraic framework for the calculus of divided differences \cite{joni1979coalgebras}. 
Variants of this notion, such as balanced infinitesimal bialgebras (also termed antisymmetric infinitesimal bialgebras   or associative D-bialgebras in different contexts \cite{aguiar2001associative, bai2010double, zhelyabin1997jordan}), have been systematically studied by Aguiar \cite{aguiar2000infinitesimal, aguiar2001associative}, who established their close analogy with Lie bialgebras. 
These structures have  been found significant applications in combinatorial mathematics.

Building on these ideas, a unified bialgebra theory for Poisson algebras, called Poisson bialgebras, was later developed in \cite{ni2013poisson}, combining aspects of both Lie and  infinitesimal bialgebras.

Meanwhile, in the realm of Lie bialgebras, quasi-triangular Lie bialgebra structures have been pivotal objects in mathematical physics \cite{drinfeld1986quantum, semenov1983classical}. 
Among these quasi-triangular Lie bialgebra structures, factorizable Lie bialgebras constitute a particularly important subclass, linking classical 
$r$-matrices to factorization problems and playing a key role in integrable systems \cite{bai2010nonabelian, reshetikhin1988quantum, semenov2003integrable}. 
Recently, quasi-triangular and factorizable theories has been extended to antisymmetric infinitesimal bialgebras with \cite{sheng2023quasi} introducing quasi-triangular and factorizable antisymmetric infinitesimal bialgebras.

Naturally, we try to establish the quasi-triangular and factorizable theories in the context of Poisson bialgebras, synthesizing concepts from both quasi-triangular Lie bialgebras and quasi-triangular antisymmetric infinitesimal bialgebras.
More precisely, we introduce the notion of quasi-triangular Poisson bialgebras based on the $(\ad, L)$-invariant condition.
In particular, if the symmetric part of the solution of the Poisson Yang-Baxter equation in a quasi-triangular Poisson bialgebra is nondegenerate, then a factorizable Poisson bialgebra is obtained.
We prove that every factorizable Poisson bialgebra induces a factorization of its underlying Poisson algebra.
Furthermore, we establish that the Drinfeld classical double of a Poisson bialgebra is automatically endowed with a canonical factorizable Poisson bialgebra structure.

Recent results have shown that factorizable Lie bialgebras and factorizable antisymmetric infinitesimal bialgebras could be characterized by quadratic Rota-Baxter Lie algebras of nonzero weights and symmetric Rota-Baxter Frobenius algebras of nonzero weights respectively (\cite{lang2023factorizable, sheng2023quasi}).
This motivates our investigation of analogous Rota-Baxter characterizations of factorizable Poisson bialgebras.
For this purpose, we introduce the notion of a quadratic Rota-Baxter Poisson algebra by equipping a quadratic Poisson algebra with a Rota-Baxter operator satisfying a compatibility condition. We show that a quadratic Rota-Baxter Poisson algebra of zero weight can give rise to a triangular Poisson bialgebra. Moreover, we establish a one-to-one correspondence between factorizable Poisson bialgebras and quadratic Rota-Baxter Poisson algebras of nonzero weights.

Building on the fact that a Poisson algebra can be obtained from a commutative differential algebra with two commuting derivations (\cite{bhaskara1988poisson}), \cite{lin2023differential} extends such a connection to the context of bialgebras, utilizing the theory of differential antisymmetric infinitesimal (ASI) bialgebras to construct Poisson bialgebras from commutative and cocommutative differential ASI bialgebras.
In this paper, we further investigate this relationship in greater depth. 
Specifically, we develop the theories of quasi-triangular and factorizable differential  ASI   bialgebras, and apply them to the study of their Poisson bialgebra counterparts. 
We establish the constructions of quasi-triangular and factorizable Poisson bialgebras from quasi-triangular and factorizable (commutative and cocommutative) differential  ASI  bialgebras respectively.

The paper is organized as follows. 
In Section~\ref{sec:qtpb}, we introduce the notion of quasi-triangular (triangular) Poisson bialgebras as a special class of coboundary Poisson bialgebras. 
% We also show that every quadratic Poisson algebra induces an isomorphism between the adjoint and coadjoint representations.
Section~\ref{sec:fpb} presents the concept of factorizable Poisson bialgebras, which is a distinguished subclass of quasi-triangular Poisson bialgebras.
We demonstrate that a factorizable Poisson bialgebra induces a factorization of the underlying Poisson algebra.
Furthermore, we prove that the Drinfeld classical double of a Poisson bialgebra is naturally endowed with a factorizable Poisson bialgebra structure.
Section~\ref{sec:rbcfpb} establishes the Rota-Baxter characterization of factorizable Poisson bialgebras. 
We introduce the notion of quadratic Rota-Baxter Poisson algebras and establish a one-to-one correspondence between quadratic Rota-Baxter Poisson algebras of nonzero weights and factorizable Poisson bialgebras.
Moreover, we show that a quadratic Rota-Baxter Poisson algebra of zero weight can give rise to a triangular Poisson bialgebra.
In Section~\ref{sec:qpbfqdasi}, we introduce the notions of quasi-triangular and factorizable differential ASI bialgebras,  and give Rota-Baxter characterization of factorizable differential ASI bialgebras. 
The constructions of quasi-triangular and factorizable Poisson bialgebras from quasi-triangular and factorizable (commutative and cocommutative) differential ASI bialgebras are illustrated respectively.

Throughout this paper, we work over a base field $\mathbb{K}$  of characteristic $0$, and all vector spaces and algebras are assumed to be finite-dimensional. 
We adopt the following conventions and  notations.
\begin{enumerate}
	\item 
	Let $(A, \diamond)$ be a vector space equipped with a bilinear operation $\diamond: A \otimes A \to A$.
	Let $L_\diamond(a)$ and $R_\diamond(a)$ denote the left and right multiplication operator respectively, that is 
	\begin{equation*}
		L_\diamond(a)b = R_\diamond(b)a = a \diamond b, \;\; \forall a, b \in A.
	\end{equation*}
	We also simply denote them by $L(a)$ and $R(a)$ respectively without confusion.
	In particular, if $(A, [\ ,\ ])$ is a Lie algebra, we let $\ad_{[,]}(a) = \ad(a)$ denote the adjoint operator, that is 
	\begin{equation*}
		\ad_{[,]}(a)b = \ad(a) b = [a, b], \;\; \forall a, b \in A.
	\end{equation*}
	
	\item 
	Let  $V$ be a vector space. 
	Denote the flip operator by $\tau: V\otimes V\rightarrow V\otimes V$, which is defined by 
	\begin{equation*}
		\tau(u\otimes v) = v \otimes u, \;\; \forall u, v \in V.
	\end{equation*}
	
	\item 
	Let $(A, \diamond)$ be a vector space equipped with a bilinear operation $\diamond: A \otimes A \to A$.
	Let $r = \sum_{i} a_i \otimes b_i \in A \otimes A$.
	Set 
	\begin{equation*}
		r_{12} = \sum_{i} a_i \otimes b_i \otimes 1, \;\;
		r_{13} = \sum_{i} a_i \otimes 1 \otimes b_i, \;\;
		r_{23} = \sum_{i} 1 \otimes a_i \otimes b_i,
	\end{equation*}
	where $1$ is the unit if $(A, \diamond)$ is unital or a symbol playing a similar role as the unit for the non-unital cases.
	Furthermore, the notations $r_{12} \diamond r_{13}, r_{13} \diamond r_{23}, r_{23} \diamond r_{12}$ are defined by 
	\begin{align*}
		&r_{12} \diamond r_{13} = \sum_{i, j} a_i \diamond a_j \otimes b_i \otimes b_j, \;\;
		r_{13} \diamond r_{23} = \sum_{i,j} a_i \otimes a_j \otimes b_i \diamond b_j, \\
		&r_{23} \diamond r_{12} = \sum_{i,j} a_i \otimes a_j \diamond b_i \otimes b_j, \;\; r_{12} \diamond r_{23} = \sum_{i,j} a_i \otimes b_i \diamond a_j \otimes b_j.
	\end{align*}
	\item 
	Denote the standard pairing between the dual space $V^*$ and $V$ by 
	\begin{equation*}
		\langle \ ,\ \rangle: V^* \times V \to \mathbb{K}, \;\; \langle f, v \rangle := f(v), \;\; \forall f \in V^*, \; v \in V.
	\end{equation*} 
	
	\item 
	Let $V, W$ be two vector spaces and $T: V \to W$ be a linear map.
	Denote the dual map by $T^*: W^* \to V^*$, which is defined by 
	\begin{equation*}
		\langle v, T^*(w^*)\rangle = \langle T(v), w^*\rangle, \;\; \forall v \in V, w^* \in W^*.
	\end{equation*}

	\item 
	Let $A,V$ be  vector spaces.
	For a linear map $\mu: A \to \End(V)$, define a linear map $\mu^*: A \to \End(V^*)$ by 
	\begin{equation*}
		\langle \mu^*(a)v^*, u\rangle = \langle v^*, \mu(a) u\rangle, \;\; \forall a \in A, u \in V, v^* \in V^*,
	\end{equation*}
	that is, $\mu^*(a) = (\mu(a))^*$.
	
	\item 
	Let $\Pi_1 = \{\alpha_k: V_1 \to V_1\}_{k=1}^m$ and $\Pi_2 = \{\beta_k: V_2 \to V_2\}_{k=1}^m$ be two sets of commuting linear maps,
	then obviously $\{\alpha_k + \beta_k\}_{k=1}^m$ is still a set of commuting linear maps, which is denoted by $\Pi_1 + \Pi_2$.
	
\end{enumerate}

%%%%%%%%%%%%%%%%%%%%%%%%%%%%%%%%%%%%%%%%%%%%%%%%%%%%%%%%%%%%%%%%%%%%%%%%%%%%%%%%
%%%%%%%%%%%%%%%%%%%%%%%%%%%%%%%%%%%%%%%%%%%%%%%%%%%%%%%%%%%%%%%%%%%%%%%%%%%%%%%%
%%%%%%%%%%%%%%%%%%%%%%%%%%%%%%%%%%%%%%%%%%%%%%%%%%%%%%%%%%%%%%%%%%%%%%%%%%%%%%%%
\section{Quasi-triangular Poisson bialgebras}\label{sec:qtpb}
In this section, we recall the notion of coboundary Poisson bialgebras and introduce the notion of quasi-triangular Poisson bialgebras as a special case of the coboundary Poisson bialgebras, based on the notion of $(\ad, L)$-invariance. The notion of a quadratic Poisson algebra is also introduced, which gives rise to a $(\ad, L)$-invariant 2-tensor and serves as the foundation for the subsequent notion of a quadratic Rota-Baxter Poisson algebra.

\begin{definition}(\cite{ni2013poisson})
	Let $(A, [\ ,\ ]_A, \cdot_A)$ be a Poisson algebra, $V$ be a vector space and $\rho, \mu: A \to \End (V)$ be two linear maps.
	Then the triple $(V, \rho, \mu)$ is called a \textbf{representation} of the Poisson algebra $(A, [\ ,\ ]_A, \cdot_A)$
	if the following conditions hold:
	\begin{enumerate}
		\item $(V, \rho)$ is a representation of the Lie algebra $(A, [\ ,\ ]_A)$, that is, 
		$\rho([a,b]_A) = \rho(a)\rho(b) - \rho(b)\rho(a)$ for all $a, b \in A$.
		\item $(V, \mu)$ is a representation of $(A, \cdot_A)$, that is, 
		$\mu(a\cdot_A b) = \mu(a)\mu(b)$ for all $a, b \in A$.
		\item 
		The following equations hold:
		\begin{align*}
			\rho(a \cdot_A b) &= \mu(b) \rho(a) + \mu(a) \rho(b),  \\
			\mu([a, b]_A) &= \rho(a) \mu(b) - \mu(b) \rho(a), \;\; \forall a,b\in A. 
		\end{align*}	
	\end{enumerate}
\end{definition}

\begin{proposition}
	{\rm (\cite{ni2013poisson}) }
	Let $(A, [\ ,\ ]_A, \cdot_A)$ be a Poisson algebra, $V$ be a vector space and $\rho, \mu: A \to \End (V)$ be two linear maps. 
	Then $(V,\rho, \mu)$ is a representation
	of $(A, [\ ,\ ]_A, \cdot_A)$ if and only if there is a Poisson algebra
	structure on $A\oplus V$ with the bilinear operations $[\ ,\ ]$ and $\cdot$ respectively defined by
	\begin{align*}
		\left[a_{1}+v_{1}, a_{2}+v_{2}\right] &=
		[a_{1},a_{2}]_A+\rho(a_1)v_2 - \rho(a_2)v_1, \\
		\left(a_{1}+v_{1}\right) \cdot \left(a_{2}+v_{2}\right) &= a_{1} \cdot_A a_{2}+\mu\left(a_{1}\right) v_{2}+\mu\left(a_{2}\right) v_{1},\; \forall a_{1},
		a_{2} \in A, v_{1}, v_{2} \in V.
	\end{align*}
	In this case, the resulting Poisson algebra
	structure on $A\oplus V$ is denoted by $(A \ltimes_{\rho,\mu}V, [\ ,\ ], \cdot)$ and called the \textbf{semi-direct product Poisson algebra} by $(A, [\ ,\ ]_A, \cdot_A)$ and $(V, \rho, \mu)$.
\end{proposition}

\begin{example}
	(\cite{ni2013poisson})
	Let $(A, [\ ,\ ]_A, \cdot_A)$ be a Poisson algebra. 
	Then $(A, \ad, L)$ is a representation of $(A, [\ ,\ ]_A, \cdot_A)$, 
	which is called the \textbf{adjoint representation}, 
	and $(A^*, -\ad^*, L^*)$ is also a representation of $(A, [\ ,\ ]_A, \cdot_A)$, which is called the \textbf{coadjoint representation}.  
\end{example}

A \textbf{Lie bialgebra} is a pair of Lie algebras $(A, [\ ,\ ]_A)$ and $(A^*, [\ ,\ ]_{A^*})$ such that 
\begin{equation*}
	\delta([a, b]_A) = (\ad(a) \otimes \id + \id \otimes \ad(a) )\delta(b) - (\ad (b) \otimes \id + \id \otimes \ad (b) )\delta(a), \;\;
	\forall a, b \in A,
\end{equation*}
where $\delta: A \to     \mathrm{Alt}^2 (A)$   is defined by $\langle \delta(a), x^* \otimes y^*\rangle = \langle a, [x^*, y^*]_{A^*} \rangle$.
A Lie bialgebra is denoted by $((A, [\ ,\ ]_A), (A^*, [\ ,\ ]_{A^*}))$ or $(A, [\ ,\ ]_A, \delta)$.

An \textbf{infinitesimal bialgebra} is a pair of  associative algebras $(A, \cdot_A)$ and $(A^*, \cdot_{A^*})$ such that 
\begin{equation*}
	\Delta(a \cdot_A b) = (L(a) \otimes \id)\Delta(b) + (\id \otimes R(b) )\Delta(a), \;\;
	\forall a, b \in A,
\end{equation*}
where $\Delta: A \to A \otimes A$ is defined by $\langle \Delta(a), x^* \otimes y^*\rangle = \langle a, x^* \cdot_{A^*} y^* \rangle$.
An infinitesimal bialgebra is denoted by $((A, \cdot_A), (A^*, \cdot_{A^*}))$ or $(A, \cdot_A, \Delta)$.

\begin{definition}
	(\cite{ni2013poisson}) 
	A \textbf{Poisson bialgebra} is a pair of Poisson algebras $(A, [\ ,\ ]_A, \cdot_A)$ and $(A^*$, $[\ ,\ ]_{A^*}$, $\cdot_{A^*})$ such that 
	\begin{enumerate}
		\item 
		$((A, [\ ,\ ]_A), (A^*, [\ ,\ ]_{A^*}))$ is a Lie bialgebra;
		
		\item 
		$((A, \cdot_A), (A^*, \cdot_{A^*}))$ is an infinitesimal bialgebra;
		
		\item 
		$\delta$ and $\Delta$ are compatible in the following sense:
		\begin{align*}
			\delta(a \cdot_A b) &= (L(a) \otimes \id)\delta(b) + (L(b) \otimes \id)\delta(a) + (\id \otimes \ad(a))\Delta(b) +(\id \otimes \ad(b))\Delta(a),\\
			%%%%%%%%%%%%%%%%%%%%%%%%%
			\Delta([a, b]_A) &= (\ad(a) \otimes \id + \id \otimes \ad(a))\Delta(b) + (L(b) \otimes \id - \id \otimes L(b))\delta(a), \;\; \forall a, b \in A,
		\end{align*}
		where $\delta, \Delta$ are the linear duals of $[\ ,\ ]_{A^*}$ and $\cdot_{A^*}$ respectively.
	\end{enumerate}
	A Poisson bialgebra is denoted by $((A, [\ ,\ ]_A, \cdot_A), (A^*, [\ ,\ ]_{A^*}, \cdot_{A^*}))$ or $(A, [\ ,\ ]_A, \cdot_A, \delta, \Delta)$.
\end{definition}
Poisson bialgebras could be equivalently characterized by Manin triples of Poisson algebras (\cite{ni2013poisson}). 
Notably, given a Poisson bialgebra $((A, [\ ,\ ]_A, \cdot_A), (A^*, [\ ,\ ]_{A^*}, \cdot_{A^*}))$, the pair $( (A^*, [\ ,\ ]_{A^*}$, $\cdot_{A^*})$, $(A, [\ ,\ ]_A, \cdot_A))$ forms a Poisson bialgebra.

\begin{definition}
	Let $(A, [\ ,\ ]_A, \cdot_A, \delta_A, \Delta_A)$ and $(B, [\ ,\ ]_B, \cdot_B, \delta_B, \Delta_B)$ be two Poisson bialgebras.
	A \textbf{homomorphism} of Poisson bialgebras is a homomorphism of Poisson algebras $\varphi: A \to B$ such that 
	\begin{equation*}
		(\varphi \otimes \varphi)\delta_A = \delta_B \circ \varphi, \;\; 
		(\varphi \otimes \varphi)\Delta_A = \Delta_B \circ \varphi.
	\end{equation*}
	Furthermore, if $\varphi: A \to B$ is a linear isomorphism of vector spaces, then $\varphi: A \to B$ is called an \textbf{isomorphism} of Poisson bialgebras.
\end{definition}

\begin{proposition}\label{prop:tmpb}
	Let $((A, [\ ,\ ]_A, \cdot_A), (A^*, [\ ,\ ]_{A^*}, \cdot_{A^*}))$ be a Poisson bialgebra and $B$ be a vector space.
	Suppose that $\varphi: A \to B$ is a linear isomorphism of vector spaces.
	Define brackets $[\ ,\ ]_B: B \otimes B \to B, \; [\ , \ ]_{B^*}: B^* \otimes B^* \to B^*$ and multiplications $\cdot_B: B \otimes B \to B, \; \cdot_{B^*}: B^* \otimes B^* \to B^*$ respectively by 
	\begin{align*}
		&[a, b]_B = \varphi([\varphi^{-1}(a), \varphi^{-1}(b)]_A), \;\; && a \cdot_B b = \varphi(\varphi^{-1}(a) \cdot_A \varphi^{-1}(b)), \;\; \forall a, b \in B, \\
		&[x^*, y^*]_{B^*} = (\varphi^*)^{-1}([\varphi^{*}(x^*), \varphi^{*}(y^*)]_{A^*}), \;\; &&x^* \cdot_{B^*} y^* = (\varphi^*)^{-1}(\varphi^{*}(x^*) \cdot_{A^*} \varphi^{*}(y^*)), \;\;\forall x^*, y^* \in B^*. 
	\end{align*}
	Then $((B, [\ ,\ ]_B, \cdot_B), (B^*, [\ ,\ ]_{B^*}, \cdot_{B^*}))$ is a Poisson bialgebra.
	Furthermore, $\varphi$ is an isomorphism of Poisson bialgebras.
\end{proposition}
\begin{proof}
	It follows from a straightforward verification.
\end{proof}

\begin{definition}
	(\cite{ni2013poisson})
	A Poisson bialgebra $(A, [\ ,\ ]_A, \cdot_A, \delta, \Delta)$ is called \textbf{coboundary}
	if there exists $r \in A \otimes A$ such that:
	\begin{align}
		\delta(a) &= (\id \otimes \ad(a) + \ad(a) \otimes \id )(r), \label{eq:pcbd1}\\
		\Delta(a) &= (\id \otimes L(a) - L(a) \otimes \id)(r), \;\; \forall a \in A. \label{eq:pcbd2}
	\end{align}
	A coboundary Poisson bialgebra is denoted by $(A, [\ ,\ ]_A, \cdot_A, \delta_r, \Delta_r)$.
\end{definition}

\begin{theorem}\label{thm:cbd}
	{\rm (\cite[Theorem~2]{ni2013poisson})}
	Let $(A, [\ ,\ ]_A, \cdot_A)$ be a Poisson algebra and $r \in A \otimes A$.
	Writing $r$ as $r = S + \Lambda$ with $S \in \mathrm{Sym}^2(A)$ and $\Lambda \in \mathrm{Alt}^2(A)$.
	Let $\delta: A \to A \otimes A$ and $\Delta: A \to A \otimes A$ be two linear maps defined by Eqs.~\eqref{eq:pcbd1} and \eqref{eq:pcbd2} respectively.
	Then $(A^*, \delta^*, \Delta^*)$ is a Poisson algebra such that $(A, [\ ,\ ]_A, \cdot_A, \delta, \Delta)$ is a Poisson bialgebra if and only the following conditions are satisfied (for all $a \in A$):
	\begin{enumerate}[$(1)$]
		\item 
		$(\ad(a) \otimes \id + \id \otimes \ad(a))S = 0$;
		
		\item 
		$(L(a) \otimes \id - \id \otimes L(a))S = 0$;
		
		\item 
		$(\ad(a) \otimes \id \otimes \id + \id \otimes \ad(a) \otimes \id + \id \otimes \id \otimes \ad(a))\bm{C}(r) = 0$;
		
		\item 
		$(L(a) \otimes \id \otimes \id - \id \otimes \id \otimes L(a))\bm{A}(r) = 0$;
		
		\item 
		$(\ad(a) \otimes \id \otimes \id)\bm{A}(r) - (\id \otimes L(a) \otimes \id - \id \otimes \id \otimes L(a))\bm{C}(r) = 0$,
	\end{enumerate}
	where $\bm{C}(r) := [r_{12}, r_{13}]_A + [r_{13}, r_{23}]_A + [r_{12}, r_{23}]_A$ and $\bm{A}(r) := r_{12} \cdot_A r_{13} + r_{13} \cdot_A r_{23} - r_{23} \cdot_A r_{12}$.
\end{theorem}

\begin{definition}
	(\cite{ni2013poisson})
	Let $(A, [\ ,\ ]_A, \cdot_A)$ be a Poisson algebra and $r \in A \otimes A$.
	Then $r$ is called a solution of the \textbf{Poisson Yang-Baxter equation} in $(A, [\ ,\ ]_A, \cdot_A)$ if 
	$\bm{C}(r) = \bm{A}(r) = 0$.
\end{definition}
Note that $\bm{C}(r) = 0$ is called the \textbf{classical Yang-Baxter equation} in the Lie algebra $(A, [\ ,\ ]_A)$ and $\bm{A}(r) = 0$ is called the \textbf{associative Yang-Baxter equation} in the associative algebra $(A, \cdot_A)$.

Let $A$ be  a vector space. Any $r \in A \otimes A$ can be identified as maps from $A^*$ to $A$, which we denote by $r_{+}: A^* \to A$ and $r_{-}: A^* \to A$, respectively and explicitly,
\begin{equation*}
	\langle r_{+}(x^*), y^*\rangle = - \langle x^*, r_{-}(y^*)\rangle = \langle r, x^* \otimes y^*\rangle, \;\; \forall x^*, y^* \in A^*.
\end{equation*}
Note that $(\tau(r))_{+} = -r_{-}$ and $(\tau(r))_{-} = -r_{+}$. 
Moreover, the bracket and multiplication on $A^*$ defined by Eqs.~\eqref{eq:pcbd1}-\eqref{eq:pcbd2} (as the duals) are respectively given by
\begin{align}
	[x^*, y^*]_r &= -\ad^*(r_{+}(x^*)) y^* +\ad^*(r_{-}(y^*)) x^* , \label{eq:pcbdr1}\\
	x^* \cdot_r y^* &= L^*(r_{+}(x^*)) y^* + L^*(r_{-}(y^*)) x^*, \;\; \forall x^*,y^* \in A^*. \label{eq:pcbdr2}
\end{align}

\begin{lemma}\label{Dj:randtr}
	Let $(A, [\ ,\ ]_A, \cdot_A)$ be a Poisson algebra and $r \in A \otimes A$.  
	Then the following statements are equivalent:
 	\begin{enumerate}[$(1)$]
		\item\label{pbr:1} 
		$r$ is a solution of the Poisson Yang-Baxter equation in $(A, [\ ,\ ]_A, \cdot_A)$.
		
		\item\label{pbr:2} 
		$\tau (r)$ is a solution of the Poisson Yang-Baxter equation in $(A, [\ ,\ ]_A, \cdot_A)$.
		
		\item\label{pbr:3} 
		The following equations hold:
		\begin{align}
			[ r_{+}\left(x^{*}\right) , r_{+}\left(y^{*}\right)]_A &=r_{+}\left(-\operatorname{ad}^{*}\left(r_{+}\left(x^{*}\right)\right) y^{*}+\operatorname{ad}^{*}\left(r_{-}\left(y^{*}\right)\right) x^{*}\right), \label{bb1}\\
			%%%%%%%%%%%%%%%%%%%%%%%%%%%%%%%%%%%%%%%
			r_{+}\left(x^{*}\right) \cdot_A r_{+}\left(y^{*}\right)&=r_{+}\left(L^{*}\left(r_{+}\left(x^{*}\right)\right) y^{*}+L^{*}\left(r_{-}\left(y^{*}\right)\right) x^{*}\right), \quad \forall x^{*}, y^{*} \in A^{*}. \label{bb2}
			\end{align} 
		
		\item\label{pbr:4} 
		The following equations hold:
		\begin{align}
			[ r_{-}\left(x^{*}\right) , r_{-}\left(y^{*}\right)]_A &= r_{-}\left(-\operatorname{ad}^{*}\left(r_{-}\left(x^{*}\right)\right) y^{*}+\operatorname{ad}^{*}\left(r_{+}\left(y^{*}\right)\right) x^{*}\right),\\
			%%%%%%%%%%%%%%%%%%%%%%%%%%%%%%%%%%%%
			r_{-}\left(x^{*}\right) \cdot_A r_{-}\left(y^{*}\right)&=r_{-}\left(L^{*}\left(r_{-}\left(x^{*}\right)\right) y^{*}+L^{*}\left(r_{+}\left(y^{*}\right)\right) x^{*}\right), \quad \forall x^{*}, y^{*} \in A^{*}. 
		\end{align}
	 \end{enumerate}
\end{lemma}
\begin{proof}
	(\ref{pbr:1}) $ \Longleftrightarrow $ (\ref{pbr:2}): Let $r = \sum\limits_i a_i \otimes b_i \in A \otimes A$. Then we have
	\begin{align*}
		\bm{C}(\tau(r)) &=  \sum_{i,j}  [b_{i}, b_{j}]_A  
		 \otimes   a_{i}   \otimes  a_{j}+ b_{i}  \otimes  [a_{i}, b_{j}]_A \otimes    a_{j}  + b_{i}  \otimes 
		b_{j} \otimes   [a_{i}, a_{j}]_A \\
		&=  \sigma_{13}\left(\sum_{i,j} a_{j}   
		 \otimes   a_{i}   \otimes  [b_{i}, b_{j}]_A +  a_{j}\otimes  [a_{i}, b_{j}]_A \otimes     b_{i}  +  [a_{i}, a_{j}]_A \otimes 
		b_{j} \otimes    b_{i}\right)\\
		&=  -\sigma_{13}(\bm{C}(r)),\\
		\bm{A}(\tau(r)) &=  \sum_{i,j}   b_{i} \cdot_A b_{j}  
		 \otimes   a_{i}   \otimes  a_{j}  + b_{i}  \otimes 
		b_{j} \otimes    a_{i} \cdot_A   a_{j} - b_{i}  \otimes     b_{j} \cdot_A a_{i}  \otimes    a_{j}\\
		&=  \sigma_{13}\left(\sum_{i,j}   
		 a_{j}\otimes   a_{i}   \otimes    b_{i} \cdot_A b_{j} +   a_{i} \cdot_A   a_{j}\otimes 
		b_{j} \otimes   b_{i}  -   a_{j}\otimes     b_{j} \cdot_A a_{i}  \otimes    b_{i}\right)\\
		&=  \sigma_{13}\left(\sum_{i,j}   
		 a_{j}\otimes   a_{i}   \otimes   b_{j} \cdot_A  b_{i} +   a_{j} \cdot_A   a_{i}\otimes 
		b_{j} \otimes   b_{i}  -   a_{j}\otimes     a_{i} \cdot_A  b_{j} \otimes    b_{i}\right)\\
		&=   \sigma_{13}(\bm{A}(r)),
	\end{align*}
	where $\sigma_{13} \in \operatorname{End}(A \otimes A \otimes A)$ is defined by  $\sigma_{13} (a \otimes b \otimes c) = c \otimes b \otimes a$ for all $a,b,c \in A$. 
	Thus, $r$ is a solution of the Poisson Yang-Baxter equation if and only if $\tau (r)$ is a solution of the Poisson Yang-Baxter equation.

	(\ref{pbr:1}) $ \Longleftrightarrow $ (\ref{pbr:3}): It follows from \cite[Proposition~3.8]{bai2010nonabelian} and  \cite[Theorem~3.5]{bai2012oass}. 

	(\ref{pbr:2}) $ \Longleftrightarrow $ (\ref{pbr:4}): It is similar to the proof of (\ref{pbr:1}) $ \Longleftrightarrow $ (\ref{pbr:3}).  
\end{proof}

We now turn to the definition of  quasi-triangular Poisson  bialgebras  as a special case of the coboundary Poisson bialgebras, and the notion of $(\ad, L)$-invariance of a 2-tensor in $A \otimes A$ is the main ingredient employed, which is motivated by Theorem~\ref{thm:cbd}.

\begin{definition}
	Let $(A, [\ ,\ ]_A, \cdot_A)$ be a Poisson algebra and $r \in A \otimes A$.
	Then $r$ is called \textbf{$(\ad, L)$-invariant} if
	\begin{align}
		(\ad(a) \otimes \id + \id \otimes \ad(a))(r) &= 0, \label{eq:adiv} \\
		(L(a) \otimes \id - \id \otimes L(a))(r) &= 0, \;\; \forall a \in A. \label{eq:liv}
	\end{align}
\end{definition}
Obviously, if $r$ is $(\ad, L)$-invariant, then $\tau(r)$ is also $(\ad, L)$-invariant. 
Denote by $I_r$ the operator
\begin{equation}
	I_r = r_{+} - r_{-}: A^* \to A. \label{eq:i}
\end{equation}
Note that $r_{-} = -r_{+}^*$ and hence $I_r^* = I_r$.
Moreover, let $S$ be the symmetric part of $r$, then $S_{+} = \frac{1}{2}I_r = \frac{1}{2}I_{\tau(r)}$.
In particular, if $r$ is antisymmetric, then $I_r=0$.

\begin{definition}\label{quasipb}
	Let $(A, [\ ,\ ]_A, \cdot_A)$ be a Poisson algebra. 
	If $r$ is a solution of the Poisson Yang-Baxter equation in $(A, [\ ,\ ]_A, \cdot_A)$ and the symmetric part of $r \in A \otimes  A$ is $(\ad, L)$-invariant, then the coboundary Poisson bialgebra $(A, [\ ,\ ]_A, \cdot_A, \delta_r,  \Delta_r)$ induced by $r$ is called a \textbf{quasi-triangular Poisson bialgebra}.
	In particular, if $r$ is antisymmetric, then $(A, [\ ,\ ]_A, \cdot_A, \delta_r, \Delta_r)$ is called a \textbf{triangular Poisson bialgebra}.
\end{definition}

\begin{proposition}\label{tquasipba}
	Let $(A, [\ ,\ ]_A, \cdot_A)$ be a Poisson algebra and $r \in A \otimes  A$. 
	Then $(A, [\ ,\ ]_A, \cdot_A, \delta_r, \Delta_r)$ is a  quasi-triangular Poisson bialgebra if and only if $(A, [\ ,\ ]_A, \cdot_A, \delta_{\tau(r)}, \Delta_{\tau(r)})$ is a quasi-triangular Poisson bialgebra.
\end{proposition}
\begin{proof}
	It follows from Lemma~\ref{Dj:randtr}.
\end{proof}

\begin{lemma}\label{lem:adi}
	Let $(A, [\ ,\ ]_A, \cdot_A)$ be a Poisson algebra and $r \in A \otimes A$.
	Let $S$ be the symmetric part of $r$. 
	Then the following conditions are equivalent:
	\begin{enumerate}[$(1)$]
		\item\label{pbinv:1} 
		$S$ is $(\ad, L)$-invariant. 
		
		\item\label{pbinv:2}  
		The following equations hold:
		\begin{align}
			S_{+}(\ad^*(a) x^*) + [a, S_{+}(x^*)]_A &= 0, \label{eq:adi1} \\ 
			%%%%%%%%%%%%%%%%%%%%%%%%%%%%%
			S_{+}(L^*(a) x^*) - a \cdot_A S_{+}(x^*)&= 0, \;\; \forall a \in A, x^* \in A^*. \label{eq:adi2} 
		\end{align}
		
		\item\label{pbinv:3} 	
		The following equations hold:
		\begin{align}
			\ad^*(S_+(x^*))y^* + \ad^*(S_+(y^*))x^* &= 0, \label{eq:sadi1} \\
			%%%%%%%%%%%%%%%%%%%%%%
			L^*(S_+(x^*))y^* - L^*(S_+(y^*))x^* &= 0, \;\; \forall  x^*,y^* \in A^*. \label{eq:sadi2}
		\end{align}
	\end{enumerate}
\end{lemma}
\begin{proof}
(\ref{pbinv:1}) $\Longleftrightarrow$ (\ref{pbinv:2}): For all $a \in A$ and $x^*, y^* \in A^*$, we have 
	\begin{align*}
		\langle S_{+}(\ad^*(a) x^*) + [a, S_{+}(x^*)], y^*\rangle &= \langle S, \ad^*(a) x^* \otimes y^* + x^* \otimes \ad^*(a) y^* \rangle \\
		&= \langle (\ad(a) \otimes \id + \id \otimes \ad(a))S, x^* \otimes y^* \rangle, \\
		\langle S_{+}(L^*(a) x^*) - a \cdot S_{+}(x^*), y^*\rangle &= \langle S, L^*(a) x^* \otimes y^* - x^* \otimes L^*(a) y^* \rangle \\
		&= \langle (L(a) \otimes \id - \id \otimes L(a))S, x^* \otimes y^* \rangle,
	\end{align*}
	which show  that $S$ is $(\ad, L)$-invariant if and only if Eqs.~\eqref{eq:adi1}-\eqref{eq:adi2} hold.
	
	(\ref{pbinv:2}) $\Longleftrightarrow $ (\ref{pbinv:3}): For all $a \in A$ and $x^*, y^* \in A^*$, we have 
		\begin{align*}
		\langle \ad^*(S_{+}(x^*)) y^* + \ad^*(S_{+}(y^*)) x^* , a\rangle &= \langle y^*, [S_{+}(x^*), a]_A\rangle + \langle x^*, [S_{+}(y^*), a]_A\rangle  \\
	  	 &= -\langle y^*, [ a,S_{+}(x^*)]_A + S_{+}(\ad^*(a)(x^*)) \rangle, \\
		\langle L^*(S_{+}(x^*)) y^* - L^*(S_{+}(y^*)) x^*, a \rangle &= \langle y^*, S_{+}(x^*) \cdot_A a\rangle - \langle x^*, S_{+}(y^*) \cdot_A a\rangle\\
		&= -\langle y^*,  - a \cdot_A S_{+}(x^*) + S_{+}(L^*(a)x^*)  \rangle,	
	\end{align*}
which show  that Eqs.~\eqref{eq:adi1}-\eqref{eq:adi2} hold if and only if Eqs.~\eqref{eq:sadi1}-\eqref{eq:sadi2} hold.
\end{proof}

\begin{theorem}\label{thm:pybe2h}
	Let $(A, [\ ,\ ]_A, \cdot_A)$ be a Poisson algebra and $r \in A \otimes A$.
	Writing $r$ as $r = S+ \Lambda$ with $S \in \mathrm{Sym}^2(A)$ and $\Lambda \in \mathrm{Alt}^2(A)$.
	Suppose that $S$ is $(\ad, L)$-invariant.
	Then $r$ is a solution of the Poisson Yang-Baxter equation in $(A, [\ ,\ ]_A, \cdot_A)$ if and only if $(A, [\ ,\ ]_r, \cdot_r)$ is a Poisson algebra and the linear maps $r_{+}, r_{-}: (A^*, [\ ,\ ]_r, \cdot_r) \to (A, [\ ,\ ]_A, \cdot_A)$ are both Poisson algebra homomorphisms, where $[\ ,\ ]_r, \cdot_r: A^* \otimes A^* \to A^*$ are respectively defined by Eqs.~\eqref{eq:pcbdr1} and \eqref{eq:pcbdr2}.
\end{theorem}
\begin{proof}
	($\Longrightarrow$) 
	By Definition~\ref{quasipb}, $(A, [\ ,\ ]_A, \cdot_A, \delta_r, \Delta_r)$ is a quasi-triangular Poisson bialgebra where $\delta_r,\Delta_r$ are respectively defined by Eqs.~\eqref{eq:pcbd1} and \eqref{eq:pcbd2}. 
	Thus, $(A, [\ ,\ ]_r, \cdot_r)$ is a Poisson algebra where $[\ ,\ ]_r, \cdot_r$ are respectively given by Eqs.~\eqref{eq:pcbdr1} and \eqref{eq:pcbdr2}. 
	Moreover, by Lemma~\ref{Dj:randtr} and \ref{lem:adi}, for all $x^*, y^* \in A^*$, we have 
	\begin{align*}
		r_{+}\left([x^{*} ,y^{*}]_r \right)  &=r_{+}\left(-\operatorname{ad}^{*}\left(r_{+}\left(x^{*}\right)\right) y^{*}+\operatorname{ad}^{*}\left(r_{-}\left(y^{*}\right)\right) x^{*}\right) =[ r_{+}\left(x^{*}\right) , r_{+}\left(y^{*}\right)]_A ,\\
		%%%%%%%%%%%%%%%%%%%%%%%%%%%%%%%%%%%%%
		r_{+}\left( x^{*} \cdot_r y^{*} \right) &=r_{+}\left(L^{*}\left(r_{+}\left(x^{*}\right)\right) y^{*}+L^{*}\left(r_{-}\left(y^{*}\right)\right) x^{*}\right)=  r_{+}\left(x^{*}\right) \cdot_A r_{+}\left(y^{*}\right),\\
		%%%%%%%%%%%%%%%%%%%%%%%%%%%%%%%%%%%%%
		r_{-}\left([x^{*} ,y^{*}]_r \right)  &=r_{-}\left(-\operatorname{ad}^{*}\left(r_{+}\left(x^{*}\right)\right) y^{*}+\operatorname{ad}^{*}\left(r_{-}\left(y^{*}\right)\right) x^{*}\right) \\
		&=r_{-}\left(-2\operatorname{ad}^{*}\left(S_{+}\left(x^{*}\right)\right) y^{*}-\operatorname{ad}^{*}\left(r_{-}\left(x^{*}\right)\right) y^{*}+\operatorname{ad}^{*}\left(r_{+}\left(y^{*}\right)\right) x^{*}-2\operatorname{ad}^{*}\left(S_{+}\left(y^{*}\right)\right) x^{*}\right) \\
		&=r_{-}\left(  -\operatorname{ad}^{*}\left(r_{-}\left(x^{*}\right)\right) y^{*}  +\operatorname{ad}^{*}\left(r_{+}\left(y^{*}\right)\right) x^{*}  \right) \\
		&=[ r_{-}\left(x^{*}\right) , r_{-}\left(y^{*}\right)]_A, \\
		%%%%%%%%%%%%%%%%%%%%%%%%%%%%%%%%%%%%%%%%%%%
		r_{-}\left( x^{*} \cdot_r y^{*} \right) &=r_{-}\left(L^{*}\left(r_{+}\left(x^{*}\right)\right) y^{*}+L^{*}\left(r_{-}\left(y^{*}\right)\right) x^{*}\right)\\
		&=r_{-}\left(2L^{*}\left(S_{+}\left(x^{*}\right)\right) y^{*} +  L^{*}\left(r_{-}\left(x^{*}\right)\right) y^{*}+L^{*}\left(r_{+}\left(y^{*}\right)\right) x^{*}-2L^{*}\left(S_{+}\left(y^{*}\right)\right) x^{*}\right)\\
		&=r_{-}\left(L^{*}\left(r_{-}\left(x^{*}\right)\right) y^{*}+L^{*}\left(r_{+}\left(y^{*}\right)\right) x^{*}\right)\\
		&=  r_{-}\left(x^{*}\right) \cdot_A r_{-}\left(y^{*}\right),
	\end{align*} 
	which shows that $r_{+}, r_{-}$ are both Poisson algebra homomorphisms.
	
	($\Longleftarrow$) 
	It follows from Lemma~\ref{Dj:randtr}. 
\end{proof}

Recall that a bilinear form $\mathfrak{B} \in \otimes^2 A^*$ on a Poisson algebra $(A, [\ ,\ ]_A, \cdot_A)$ is called \textbf{invariant} if 
\begin{equation*}
	\mathfrak{B}([a, b]_A, c) = \mathfrak{B}(a, [b, c]_A), \;
	\mathfrak{B}(a \cdot_A b, c) = \mathfrak{B}(a, b \cdot_A c),  
	\;\; \forall a, b, c \in A.
\end{equation*}
A \textbf{quadratic Poisson algebra} is a quadruple $(A, [\ ,\ ]_A, \cdot_A, \mathfrak{B})$ where $(A, [\ ,\ ]_A, \cdot_A)$ is a Poisson algebra and $\mathfrak{B} \in \otimes^2 A^*$ is a nondegenerate symmetric invariant bilinear form on $(A, [\ ,\ ]_A, \cdot_A)$.

\begin{proposition}
	Let $\left(A,[\ ,\ ]_A, \cdot_{A}\right)$ be a Poisson algebra. 
	Then $\left(A \ltimes_{-\mathrm{ad}^{*},L^{*}} A^{*}, [\ ,\ ], \cdot, \mathfrak{B}_d \right)$ is a quadratic Poisson algebra, where $\mathfrak{B}_d$ is defined by
	\begin{align}
		\mathfrak{B}_d\left(a+x^{*}, b+y^{*}\right)=  \left\langle  a,y^{*}\right\rangle + \left\langle b, x^{*}\right\rangle, \quad  \forall a, b \in A,\quad x^{*}, y^{*} \in A^{*}. \label{abm}
	\end{align}
\end{proposition}
\begin{proof}
	It follows from a straightforward computation.
\end{proof}

Let $A$ be a vector space and $\mathfrak{B}$ be a nondegenerate  bilinear form. 
Denote by $I_\mathfrak{B}: A^* \to A$ the induced linear isomorphism defined by
\begin{align}
	\langle I_\mathfrak{B}^{-1}(a), b\rangle := \mathfrak{B}(a, b), \quad \forall a,b \in A. \label{eq:sxrbt1}
\end{align}
Moreover, denote by $r_\mathfrak{B}\in A \otimes A $ the 2-tensor form of $I_\mathfrak{B}$,  that is, 
\begin{equation}
	\langle r_\mathfrak{B}, x^* \otimes y^*\rangle := \langle   I_\mathfrak{B}(x^*),y^*\rangle, \quad \forall  x^*,y^* \in A^*. \label{eq:sxrbt2}
\end{equation}

\begin{proposition}\label{qparinv}
	Let $(A, [\ ,\ ]_A, \cdot_A)$ be a Poisson algebra and $\mathfrak{B}$ be a nondegenerate  bilinear form.  
	Let $I_\mathfrak{B}: A^* \to A$ be the induced linear isomorphism  by $\mathfrak{B}$ and $r_\mathfrak{B} \in A \otimes A$ be the 2-tensor form of $I_\mathfrak{B}$ given by Eq.~\eqref{eq:sxrbt2}. 
	Then $(A, [\ ,\ ]_A, \cdot_A,\mathfrak{B})$ is a quadratic  Poisson algebra if and only if $r_\mathfrak{B}\in A \otimes A$ is symmetric and $(\ad, L)$-invariant.
\end{proposition}
\begin{proof}
	Suppose that $(A, [\ ,\ ]_A, \cdot_A,\mathfrak{B})$ is a quadratic  Poisson algebra. 
	For all $a, b, c \in A$, there exists $x^*,y^*,z^* \in A^*$ such that $I_\mathfrak{B}(x^*) = a, I_\mathfrak{B}(y^*) = b$ and $I_\mathfrak{B}(z^*) = c$. 
	Then we have
	\begin{equation*}
		\langle  r_\mathfrak{B}-\tau(r_\mathfrak{B}), x^* \otimes y^*\rangle  = \langle   I_{\mathfrak{B}}(x^*), I_\mathfrak{B}^{-1}(b)\rangle  - \langle   I_{\mathfrak{B}}(y^*), I_\mathfrak{B}^{-1}(a)\rangle  = \mathfrak{B}(b,a) -  \mathfrak{B}(a,b) = 0.
	\end{equation*}
	Thus, $r_\mathfrak{B}$ is symmetric. 
	Moreover, we have
	\begin{align*}
		\langle(\ad(c) \otimes \id + \id \otimes \ad(c))(r_\mathfrak{B}), x^* \otimes y^*\rangle  &= \langle r_\mathfrak{B}, \ad^*(c)(x^*) \otimes y^*\rangle + \langle r_\mathfrak{B} , x^* \otimes \ad^*(c)(y^*)\rangle  \\
		&= \langle I_\mathfrak{B}(y^*), \ad^*(c)(x^*) \rangle + \langle I_\mathfrak{B}(x^*),  \ad^*(c)(y^*)\rangle  \\
		&= \mathfrak{B}( [c,b]_A,  a ) + \mathfrak{B}([c,a]_A,  b)=0, \\
		%%%%%%%%%%%%%%%%%%%%%%%%%%%%%%%%%%%%%
		\langle	(L(c) \otimes \id - \id \otimes L(c))(r_\mathfrak{B}), x^* \otimes y^*\rangle &= \langle r_\mathfrak{B}, L^*(c)(x^*) \otimes y^*\rangle - \langle r_\mathfrak{B} , x^* \otimes L^*(c)(y^*)\rangle  \\
		&= \langle I_\mathfrak{B}(y^*), L^*(c)(x^*) \rangle - \langle I_\mathfrak{B}(x^*),  L^*(c)(y^*)\rangle  \\
		&= \mathfrak{B}( c \cdot_A b,  a ) - \mathfrak{B}( c \cdot_A  a ,  b)=0, 
	\end{align*}
	which shows that $r_\mathfrak{B}\in A \otimes A$ is $(\ad, L)$-invariant.  
	The converse statement could be proved by reversing the argument.
\end{proof}

The following theorem demonstrates that every quadratic Poisson algebra
naturally induces an isomorphism between the adjoint and coadjoint representations of the corresponding Poisson algebra.

\begin{definition}
	Let $(A,[\ ,\ ]_A, \cdot_A)$ be a Poisson algebra, and $(V_1, \rho_1, \mu_1)$ and $(V_2, \rho_2, \mu_2)$ be two representations of $(A,[\ ,\ ]_A,  \cdot_A)$. 
	A \textbf{homomorphism of representations} from $(V_1$, $\rho_1$, $\mu_1)$ to $(V_2$, $\rho_2$, $\mu_2)$ is a linear map $\varphi: V_1 \to V_2$ such that 
	\begin{equation*}
		\varphi \circ \rho_1(a) = \rho_2(a) \circ \varphi, \;
		\varphi \circ \mu_1(a) = \mu_2(a) \circ \varphi, \; \forall a \in A.
	\end{equation*}
	In particular, when $\varphi: V_1 \to V_2$ is a linear isomorphism of vector spaces, $\varphi$ is called an \textbf{isomorphism of representations} from $(V_1, \rho_1, \mu_1)$ to $(V_2, \rho_2, \mu_2)$.
\end{definition}

\begin{theorem}\label{pdualdengjia}
	Let $(A, [\ ,\ ]_A, \cdot_A)$ be a Poisson algebra.   
	If there is a bilinear form $\mathfrak{B}$ such that $(A, [\ ,\ ]_A, \cdot_A, \mathfrak{B})$ is a quadratic Poisson algebra, then the linear map $I_\mathfrak{B}^{-1}: A \to A^*$ defined by $\langle I_\mathfrak{B}^{-1}(a), b\rangle = \mathfrak{B}(a, b)$ 
	is an isomorphism from the adjoint representation $(A, \ad, L)$ to the coadjoint representation $(A^*, -\ad^*, L^*)$.
	
	Conversely, if $I^{-1}: A \to A^*$  is an isomorphism from the adjoint representation $(A, \ad, L)$ to the coadjoint representation $(A^*, -\ad^*, L^*)$, then the bilinear form $\mathfrak{B}$ defined by $\mathfrak{B}(a, b) = \langle I^{-1}(a), b\rangle $ is nondegenerate invariant on $(A, [\ ,\ ]_A, \cdot_A)$.
\end{theorem}
\begin{proof}
	Suppose that $(A, [\ ,\ ]_A, \cdot_A, \mathfrak{B})$ is a quadratic Poisson algebra. 
	Then for all $a, b, c \in A$, we have
	\begin{align*}
		\langle I_\mathfrak{B}^{-1} \ad(a)b, c\rangle &= -\mathfrak{B}([b, a]_A, c) = -\mathfrak{B}(b, [a, c]_A) = \langle -\ad^*(a) I_\mathfrak{B}^{-1}(b), c\rangle, \\
		\langle I_\mathfrak{B}^{-1} L(a) b, c\rangle &= \mathfrak{B}(a \cdot_A b, c)  = \mathfrak{B}(b \cdot_A a, c)  = \mathfrak{B}(b, a \cdot_A c)  = \langle L^*(a) I_\mathfrak{B}^{-1} (b), c\rangle,
	\end{align*}
	that is,
	\begin{equation*}
		I_\mathfrak{B}^{-1} \ad(a) = -\ad^*(a) I_\mathfrak{B}^{-1}, \;\;
		I_\mathfrak{B}^{-1} L(a) = L^*(a) I_\mathfrak{B}^{-1}.
	\end{equation*}
	Therefore, $I_\mathfrak{B}^{-1}: A \to A^*$ is an isomorphism from the adjoint representation $(A, \ad, L)$ to the coadjoint representation $(A^*, -\ad^*, L^*)$. 
	The converse could be proved by a similar argument as above.
\end{proof}

We conclude this section with a proposition that is highly valuable for obtaining Rota-Baxter characterizations of factorizable Poisson bialgebras in Section~\ref{sec:rbcfpb}.

\begin{proposition}\label{pro:pybe2h}
	Let $(A, [\ ,\ ]_A, \cdot_A,\mathfrak{B})$ be a quadratic Poisson algebra and $I_\mathfrak{B}: A^* \to A$ be the induced linear isomorphism by $\mathfrak{B}$. 
	Suppose that $ r \in A \otimes A$.
	Then $r$ is a solution of the Poisson Yang-Baxter equation in $(A, [\ ,\ ]_A, \cdot_A)$ if and only if the linear map $P:= r_{+} \circ I_\mathfrak{B}^{-1}: A \to A$ satisfies the following equations:
	\begin{align}
		[P(a),P(b)]_A &= P\Big([P(a),b]_A + [a,P(b)]_A -  [a, (I_{r} \circ I_\mathfrak{B}^{-1})(b)]_A\Big),\\
		%%%%%%%%%%%%%%%%%%%%%%%%%%%%%%%%%%%%%%%
		P(a) \cdot_A P(b) &= P\Big(P(a) \cdot_A b + a \cdot_A  P(b)  -  a \cdot_A (I_{r} \circ I_\mathfrak{B}^{-1})(b) \Big), \quad \forall a,b \in A.
	\end{align}
\end{proposition}
\begin{proof} 
By Theorem~\ref{pdualdengjia}, we have
	\begin{align*}
	I_\mathfrak{B}^{-1} \ad(a) &= -\ad^*(a) I_\mathfrak{B}^{-1},\quad 
		I_\mathfrak{B}^{-1} L(a)  = L^*(a) I_\mathfrak{B}^{-1},\quad \forall a \in A.
	\end{align*}
	For all $a,b \in A$, there exists $x^*,y^* \in A^*$ such that $I_\mathfrak{B}(x^*) = a$ and $I_\mathfrak{B}(y^*) = b$.  Then  we have
	\begin{align*}
		[P(a),P(b)]_A  &= [(r_{+} \circ I_\mathfrak{B}^{-1})(a),(r_{+} \circ I_\mathfrak{B}^{-1})(b)]_A = [ r_{+}  (x^*),  r_{+}  (y^*)]_A,\\
		%%%%%%%%%%%%%%%%%%%%%%%%%%
		P ([P(a),b]_A) &= (r_{+} \circ I_\mathfrak{B}^{-1})([(r_{+} \circ I_\mathfrak{B}^{-1})(a),b]_A) = (r_{+} \circ I_\mathfrak{B}^{-1})(\ad(r_{+} (x^*))I_\mathfrak{B}(y^*)) \\
		& =- r_{+} (\ad^*(r_{+} (x^*))  y^* ),\\
		%%%%%%%%%%%%%%%%%%%%%%%%%%
		P ([a,P(b)]_A) &= (r_{+} \circ I_\mathfrak{B}^{-1})([a,(r_{+} \circ I_\mathfrak{B}^{-1})(b)]_A) = -(r_{+} \circ I_\mathfrak{B}^{-1})(\ad( r_{+} (y^*))I_\mathfrak{B}(x^*)) \\
		& =  r_{+} (\ad^*(r_{+} (y^*))  x^* ),\\
		%%%%%%%%%%%%%%%%%%%%%%%%%%
		- P ([a, (I_{r} \circ I_\mathfrak{B}^{-1})(b)]_A) &= - (r_{+} \circ I_\mathfrak{B}^{-1})([a,(I_{r} \circ I_\mathfrak{B}^{-1})(b)]_A) =   (r_{+} \circ I_\mathfrak{B}^{-1})(\ad( I_{r} (y^*))I_\mathfrak{B}(x^*))  \\
		&= - r_{+} (\ad^*(I_{r} (y^*))  x^* ),\\
		%%%%%%%%%%%%%%%%%%%%%%%%%%%
		P(a) \cdot_A P(b) &= (r_{+} \circ I_\mathfrak{B}^{-1})(a) \cdot_A (r_{+} \circ I_\mathfrak{B}^{-1})(b) = r_{+}(x^*)\cdot_A r_{+}(y^*),\\
		%%%%%%%%%%%%%%%%%%%%%%%%%%%%%%
		P( P(a)\cdot_A b) &= (r_{+} \circ I_\mathfrak{B}^{-1})( (r_{+} \circ I_\mathfrak{B}^{-1})(a)\cdot_A  b)= (r_{+} \circ I_\mathfrak{B}^{-1})(L(r_{+}(x^*))I_\mathfrak{B}(y^*) ) \\
		&=  r_{+} (L^*(r_{+}(x^*)) y^* ),\\
		%%%%%%%%%%%%%%%%%%%%%%%%%%%%%%%%
		P(a \cdot_A P(b)) &= (r_{+} \circ I_\mathfrak{B}^{-1})(a \cdot_A (r_{+} \circ I_\mathfrak{B}^{-1})(b) )= (r_{+} \circ I_\mathfrak{B}^{-1})(L(r_{+}(y^*))I_\mathfrak{B}(x^*)) \\
		&=  r_{+} (L^*(r_{+}(y^*))x^* ),\\
		%%%%%%%%%%%%%%%%%%%%%%%%%%%%%%%%
		-  P ( a \cdot_A (I_{r} \circ I_\mathfrak{B}^{-1})(b) ) &= -(r_{+} \circ I_\mathfrak{B}^{-1})(a \cdot_A (I_{r} \circ I_\mathfrak{B}^{-1})(b) )=- (r_{+} \circ I_\mathfrak{B}^{-1})(L(I_{r}(y^*))I_\mathfrak{B}(x^*)) \\  
		&= - r_{+} (L^*(I_{r}(y^*))x^* ).
	\end{align*}
	By Eq.~\eqref{eq:i} and Lemma~\ref{Dj:randtr}, the conclusion follows.
\end{proof}

%%%%%%%%%%%%%%%%%%%%%%%%%%%%%%%%%%%%%%%%%%%%%%%%%%%%%%%%%%%%%%%%%%%%%%%%%%%%%%%%
\section{Factorizable Poisson bialgebras}\label{sec:fpb}
In this section, we introduce the notion of a factorizable Poisson bialgebra, which is a special quasi-triangular Poisson bialgebra satisfying that the map $I_r: A^* \to A$ is a linear isomorphism of vector spaces.
We show that the Drinfeld classical double of a Poisson bialgebra is naturally a factorizable Poisson bialgebra.

\begin{definition}
	A quasi-triangular Poisson bialgebra $(A, [\ ,\ ]_A, \cdot_A, \delta_r, \Delta_r)$ is called \textbf{factorizable} if the symmetric part $S$ of $r$ is nondegenerate, which means that the linear map $I_r: A^* \to A$ defined by Eq.~\eqref{eq:i} is a linear isomorphism of vector spaces.
\end{definition}

\begin{proposition}\label{xtquasipba}
	Let $(A, [\ ,\ ]_A, \cdot_A)$ be a Poisson algebra and $r \in A \otimes A$. 
	Then $(A, [\ ,\ ]_A, \cdot_A, \delta_r, \Delta_r)$ is a factorizable Poisson bialgebra if and only if $(A, [\ ,\ ]_A, \cdot_A, \delta_{\tau(r)}, \Delta_{\tau(r)})$ is a factorizable Poisson bialgebra.
\end{proposition}
\begin{proof}
	It follows from Proposition~\ref{tquasipba}.
\end{proof}

Consider the map 
\begin{equation*}
	A^* \xrightarrow{\;\;\;\;r_{+} \oplus r_{-}\;\;\;\;} A \oplus A \xrightarrow{(a, b) \mapsto a - b} A.
\end{equation*}

The following result justifies the terminology of a factorizable Poisson bialgebra.
\begin{proposition}\label{prop:fpt}
	Let $(A, [\ ,\ ]_A, \cdot_A, \delta_r, \Delta_r)$ be a factorizable Poisson bialgebra.
	Then $\mathrm{Im}(r_{+} \oplus r_{-})$ is a Poisson subalgebra of the direct sum Poisson algebra $A \oplus A$, 
	which is isomorphic to the Poisson algebra $(A^*, [\ ,\ ]_r, \cdot_r)$, where $[\ ,\ ]_r, \cdot_r: A^* \otimes A^* \to A^*$ are respectively defined by Eqs.~\eqref{eq:pcbdr1} and \eqref{eq:pcbdr2}.
	Moreover, any $a \in A$ has a unique decomposition 
	$a = a_{+} - a_{-}$ with $(a_{+}, a_{-}) \in \mathrm{Im}(r_{+} \oplus r_{-})$.
\end{proposition}
\begin{proof}
	By Theorem~\ref{thm:pybe2h}, both $r_{+}$ and $r_{-}$ are Poisson algebra homomorphisms.
	Therefore, $\mathrm{Im}(r_{+} \oplus r_{-})$ is a Poisson subalgebra of the direct sum Poisson algebra $A \oplus A$.
	Since $I_r = r_{+} - r_{-}$ is a linear isomorphism of vector spaces, it follows that the Poisson algebra $\mathrm{Im}(r_{+} \oplus r_{-})$ is isomorphic to the Poisson algebra $(A^*, [\ ,\ ]_r, \cdot_r)$.
	Moreover, we have
	\begin{equation*}
		r_{+}I_r^{-1}(a) - r_{-}I_r^{-1}(a) = (r_{+} - r_{-}) I_r^{-1}(a)= a,
	\end{equation*} 
	which shows that $a = a_{+} - a_{-}$ with $a_{+} = r_{+}I_r^{-1}(a)$ and $a_{-} = r_{-}I_r^{-1}(a)$.
	The uniqueness also follows from the fact that $I_r: A^* \to A$ is a linear isomorphism of vector spaces.
\end{proof}

Let $((A, [\ ,\ ]_A, \cdot_A), (A^*, [\ ,\ ]_{A^*}, \cdot_{A^*}))$ be an arbitrary Poisson bialgebra.
Define a bracket and a multiplication on $\mathfrak{A} = A \oplus A^*$ by
\begin{align*}
	[(a, x^*), (b, y^*)]_\mathfrak{A} &= ([a, b]_A - \ad_{[,]_{A^*}}^*(x^*)b + \ad_{[,]_{A^*}}^*(y^*)a, [x^*, y^*]_{A^*} - \ad_{[,]_A}^*(a)y^* + \ad_{[,]_A}^*(b)x^*), \\
	%%%%%%%%%%%%%%%%%%%%%%%%%%%%%%%%%%%%%%%%%%%%%%%%%%
	(a, x^*) \cdot_\mathfrak{A} (b, y^*) &= (a \cdot_A b + L_{\cdot_{A^*}}^*(x^*)b + L_{\cdot_{A^*}}^*(y^*)a, x^* \cdot_{A^*} y^* + L_{\cdot_{A}}^*(a)y^* + L_{\cdot_{A}}^*(b)x^*), 
\end{align*}
where $a, b \in A, x^*, y^* \in A^*$.
Then $(\mathfrak{A}, [\ ,\ ]_\mathfrak{A}, \cdot_\mathfrak{A})$ is a Poisson algebra, which is called the \textbf{Drinfeld classical double} of the Poisson bialgebra. 
Moreover, let $\{e_1, e_2, \cdots, e_n\}$ be a basis of $A$ and $\{e_1^*, e_2^*, \cdots, e_n^*\}$ be the dual basis of $A^*$.
Then 
\begin{equation*}
	r = \sum_{i}e_i \otimes e_i^* \in A \otimes A^* \subset \mathfrak{A} \otimes \mathfrak{A}
\end{equation*}
induces a coboundary Poisson bialgebra structure $(\mathfrak{A}, [\ ,\ ]_\mathfrak{A}, \cdot_\mathfrak{A}, \delta_r, \Delta_r)$ (see \cite[Theorem~3]{ni2013poisson} for more details).

\begin{theorem}\label{thm:spfb}
	The Poisson bialgebra $(\mathfrak{A}, [\ ,\ ]_\mathfrak{A}, \cdot_\mathfrak{A}, \delta_r, \Delta_r)$ is a quasi-triangular Poisson bialgebra. 
	Furthermore, it is factorizable.
\end{theorem}
\begin{proof}
	First, we need to verify that the symmetric part $S = \frac{1}{2}(e_i \otimes e_i^* + e_i^* \otimes e_i)$ of $r$ is $(\ad_{[,]_\mathfrak{A}}, L_{\cdot_\mathfrak{A}})$-invariant.
	For all $(x^*, a) \in \mathfrak{A}^*$, we have $S_{+}(x^*, a) = \frac{1}{2}(a, x^*) \in \mathfrak{A}$.
	Moreover, by a straightforward computation, we have 
	\begin{align*}
		[(a, x^*), S_{+}(y^*, b)]_\mathfrak{A} &= \frac{1}{2}([a, b]_A - \ad_{[,]_{A^*}}^*(x^*)b + \ad_{[,]_{A^*}}^*(y^*)a, [x^*, y^*]_{A^*} - \ad_{[,]_{A}}^*(a)y^* + \ad_{[,]_{A}}^*(b)x^*), \\
		%%%%%%%%%%%%%%%%%%%%%%%%%%%%%%%%%%%%%%%%%%%%%%%
		\ad^*_{[,]_{\mathfrak{A}}}(a, x^*)(y^*, b) &=-([x^*, y^*]_{A^*} - \ad_{[,]_{A}}^*(a)y^* + \ad_{[,]_{A}}^*(b)x^*, [a, b]_A - \ad_{[,]_{A^*}}^*(x^*)b + \ad_{[,]_{A^*}}^*(y^*)a), \\
		%%%%%%%%%%%%%%%%%%%%%%%%%%%%%%%%%%%%%%%%%%%%%%%
		(a, x^*) \cdot_\mathfrak{A} S_{+}(y^*, b) &= \frac{1}{2}(a \cdot_A b + L_{\cdot_{A^*}}^*(x^*)b + L_{\cdot_{A^*}}^*(y^*)a, x^* \cdot_{A^*} y^* + L_{\cdot_{A}}^*(a)y^* + L_{\cdot_{A}}^*(b)x^*), \\
		%%%%%%%%%%%%%%%%%%%%%%%%%%%%%%%%%%%%%%%%%%%%%%%
		L^*_{\cdot_\mathfrak{A}}(a, x^*)(y^*, b) &= (x^* \cdot_{A^*} y^* + L_{\cdot_{A}}^*(a)y^* + L_{\cdot_{A}}^*(b)x^*, a \cdot_A b + L_{\cdot_{A^*}}^*(x^*)b + L_{\cdot_{A^*}}^*(y^*)a). 
	\end{align*}
	Thus, we have
	\begin{align*}
		S_{+}(\ad^*_{[,]_\mathfrak{A}}(a, x^*)(y^*, b))  + [(a, x^*), S_{+}(y^*, b)]_\mathfrak{A} &= 0, \\
		S_{+}(L^*_{\cdot_\mathfrak{A}}(a, x^*)(y^*, b)) - (a, x^*) \cdot_\mathfrak{A} S_{+}(y^*, b) &= 0.
	\end{align*}
	Therefore, by Lemma~\ref{lem:adi}, 
	the symmetric part $S$ of $r$ is $(\ad_{[,]_\mathfrak{A}}, L_{\cdot_\mathfrak{A}})$-invariant.
	On the other hand, by \cite[Theorem~3]{ni2013poisson}, $r$ is a solution of the Poisson Yang-Baxter equation in $(\mathfrak{A}, [\ ,\ ]_\mathfrak{A}, \cdot_\mathfrak{A})$.
	Hence, the Poisson bialgebra $(\mathfrak{A}, [\ ,\ ]_\mathfrak{A}, \cdot_\mathfrak{A}, \delta_r, \Delta_r)$ is a quasi-triangular Poisson bialgebra. 
	
	Finally, we have to show $I_r$ is a linear isomorphism of vector spaces. Note that $I_r(x^*, a) = 2S_+(x^*, a) = (a, x^*)$, which implies that the linear map $I_r: \mathfrak{A}^* \to \mathfrak{A}$ is a linear isomorphism of vector spaces.  Therefore, the Poisson bialgebra $(\mathfrak{A}, [\ ,\ ]_\mathfrak{A}, \cdot_\mathfrak{A}, \delta_r, \Delta_r)$ is factorizable.
	{\iffalse
	Finally, we have to show $I_r$ is a linear isomorphism of vector spaces.
	Note that $r_{+}, r_{-}: \mathfrak{A}^* \to \mathfrak{A}$ are given by
	\begin{align*}
		r_{+}(x^*, a) = (0, x^*), \;
		r_{-}(x^*, a) = -(a, 0), \;\;
		\forall a \in A, x^* \in A^*.
	\end{align*}
	Thus, we have $I_r(x^*, a) = (a, x^*)$, \fi}
\end{proof}
%%%%%%%%%%%%%%%%%%%%%%%%%%%%%%%%%%%%%%%%%%%%%%%%%%%%%%%%%%%%%%%%%%%%%%%%%%%%%%%%
%%%%%%%%%%%%%%%%%%%%%%%%%%%%%%%%%%%%%%%%%%%%%%%%%%%%%%%%%%%%%%%%%%%%%%%%%%%%%%%%
%%%%%%%%%%%%%%%%%%%%%%%%%%%%%%%%%%%%%%%%%%%%%%%%%%%%%%%%%%%%%%%%%%%%%%%%%%%%%%%%
\section{Quadratic Rota-Baxter Poisson algebras}\label{sec:rbcfpb}
In this section, we introduce the notion of a quadratic Rota-Baxter Poisson algebra and show that a quadratic Rota-Baxter Poisson algebra of zero weight induces a triangular Poisson bialgebra. 
Moreover, we show that there is a one-to-one correspondence between factorizable Poisson bialgebras and quadratic Rota-Baxter Poisson algebras of nonzero weights.

Recall that a linear map $P: A \to A$ is called a \textbf{Rota-Baxter operator of weight $\lambda$} on a Poisson algebra $(A, [\ ,\ ]_A, \cdot_A)$ if 
\begin{align*}
	[P(a), P(b)]_A &= P([P(a), b]_A + [a, P(b)]_A + \lambda [a, b]_A), \\
	P(a) \cdot_A P(b) &= P(P(a) \cdot_A b + a \cdot_A P(b) + \lambda a \cdot_A b), 
	\;\; \forall a, b \in A.
\end{align*}
A \textbf{Rota-Baxter Poisson algebra $(A, [\ ,\ ]_A, \cdot_A, P)$ of weight $\lambda$} is a Poisson algebra $(A, [\ ,\ ]_A, \cdot_A)$ equipped with a Rota-Baxter operator $P$ of weight $\lambda$.

\begin{lemma}
	Let $(A, [\ ,\ ]_A, \cdot_A, P)$ be a Rota-Baxter Poisson algebra of weight $\lambda$. 
	Then $(A, [\ ,\ ]_P, \cdot_P)$ is a Poisson algebra, which is called the \textbf{descendent Poisson algebra} of $(A, [\ ,\ ]_A, \cdot_A, P)$, where $[\ ,\ ]_P, \cdot_P: A \otimes A \to A$ are respectively defined by 
	\begin{align}
		[a, b]_P &= [P(a), b]_A + [a, P(b)]_A + \lambda [a, b]_A, \label{eq:rdpal}\\
		%%%%%%%%%%%%%%%%%%%%%%%%%%%%%%%%%%%%%%%%%%%%%%%%%
		a \cdot_P b &= P(a) \cdot_A b + a \cdot_A P(b) + \lambda a \cdot_A b, \;\; \forall a, b \in A. \label{eq:rdpac}
	\end{align}
	Furthermore, $P$ is a Poisson algebra homomorphism from $(A, [\ ,\ ]_P, \cdot_P)$ to $(A, [\ ,\ ]_A, \cdot_A)$.
\end{lemma}
\begin{proof}
	The proof is straightforward.
\end{proof}

Equipping quadratic Poisson algebras with Rota-Baxter operators satisfying compatibility conditions, we introduce the notion of quadratic Rota-Baxter Poisson algebras.
\begin{definition}
	The triple $((A, [\ ,\ ]_A, \cdot_A), \mathfrak{B}, P)$ is called a \textbf{quadratic Rota-Baxter Poisson algebra of weight $\lambda$} if $(A, [\ ,\ ]_A, \cdot_A, \mathfrak{B})$ is a quadratic Poisson algebra and $(A, [\ ,\ ]_A, \cdot_A, P)$ is a Rota-Baxter Poisson algebra of weight $\lambda$ satisfying the following compatibility condition:
	\begin{equation}
		\mathfrak{B}(a, P(b)) + \mathfrak{B}(P(a), b) + \lambda \mathfrak{B}(a, b) = 0, \;\; \forall a, b \in A. \label{eq:qrbp}
	\end{equation}
\end{definition}

\begin{proposition}\label{PandtauP}
	Let $\left(A,[\ ,\ ]_A, \cdot_{A}, \mathfrak{B}\right)$ be a quadratic Poisson algebra and $P: A \rightarrow A$ be a linear map, then $\left((A,[\ ,\ ]_A,\cdot_{A}), \mathfrak{B}, P\right)$  is a quadratic Rota-Baxter Poisson algebra of weight $\lambda$ if and only if $((A,[\ ,\ ]_A,\cdot_{A}), \mathfrak{B},\tilde{P})$ is a quadratic Rota-Baxter Poisson algebra of weight $\lambda$, where $\tilde{P} := - \lambda \id - P$.
\end{proposition}
\begin{proof}
	It is easy to show that $P$ is a Rota-Baxter operator of weight $\lambda$ if and only if $\tilde{P}$ is a Rota-Baxter operator of weight $\lambda$. 
	Moreover, we have
	\begin{equation*}
		\mathfrak{B}(\tilde{P}(a), b)+\mathfrak{B}(a,\tilde{P}(b))+\lambda \mathfrak{B}(a, b)=-\mathfrak{B}(P(a), b)-\mathfrak{B}(a, P(b))-\lambda \mathfrak{B}(a, b), \quad \forall a,b \in A.
	\end{equation*}
	Thus, $P$ satisfies Eq.~\eqref{eq:qrbp} if and only if $\tilde{P}$ satisfies Eq.~\eqref{eq:qrbp}. 
	The proof is completed.
\end{proof}

\begin{proposition}\label{exrbfna}
	Let $\left(A,[\ ,\ ]_A, \cdot_{A}, P\right)$ be a Rota-Baxter Poisson algebra of weight $\lambda$. Then $((A \ltimes_{-\mathrm{ad}^{*},L^{*}} A^{*}, [\ ,\ ], \cdot), \mathfrak{B}_d, P+\tilde{P}^{*})$ is a quadratic Rota-Baxter Poisson algebra of weight $\lambda$, where $\mathfrak{B}_d$ is defined by Eq.~\eqref{abm} and $\tilde{P} := - \lambda \id - P$.
\end{proposition}
\begin{proof}
	The proof is straightforward.
\end{proof}

\begin{lemma}\label{rbfna0}
	Let $A$ be a vector space and $\mathfrak{B}$ be a nondegenerate symmetric bilinear form.
	Let $I_\mathfrak{B}: A^* \to A$ be the induced linear isomorphism by $\mathfrak{B}$ and $r_\mathfrak{B} \in A \otimes A$ be the 2-tensor form of $I_\mathfrak{B}$ given by Eq.~\eqref{eq:sxrbt2}. 
	Suppose that $ r \in A \otimes A$. 
	Then $P_{r}:= r_+ \circ I_\mathfrak{B}^{-1}$ satisfies Eq.~\eqref{eq:qrbp} if and only if $r+\tau(r)=-\lambda r_{\mathfrak{B}}$  where $\lambda \in \mathbb{K}$.
\end{lemma}
\begin{proof}
	For all $a, b \in A$, there exists $x^{*},y^{*} \in A^{*}$ such that $I_\mathfrak{B}\left(x^{*}\right)=a, I_\mathfrak{B}\left(y^{*}\right)=b$.
	Then
	\begin{align*}
		\mathfrak{B}\left(P_{r}(a), b\right)&=\mathfrak{B}\left(b, P_{r}(a)\right)= \left\langle I_\mathfrak{B}^{-1}(b), (r_+ \circ I_\mathfrak{B}^{-1})(a)\right\rangle=\left\langle r, x^{*} \otimes y^{*}\right\rangle,\\
		%%%%%%%%%%%%%%%%%%%%%%%%%%%%%%%%%%%%%%%%%
		\mathfrak{B}(a, P_{r}(b))&=\left\langle I_\mathfrak{B}^{-1}(a),  (r_+ \circ I_\mathfrak{B}^{-1})(b)\right\rangle=\left\langle r, y^{*} \otimes x^{*}\right\rangle=\left\langle \tau(r), x^{*} \otimes y^{*}\right\rangle, \\
		%%%%%%%%%%%%%%%%%%%%%%%%%%%%%%%%%%%%%%%%%
		\lambda \mathfrak{B}(a, b)&= \lambda \mathfrak{B}(b, a)= \lambda\left\langle I_\mathfrak{B}^{-1}(b), (I_\mathfrak{B} \circ I_\mathfrak{B}^{-1})(a)\right\rangle= \lambda\left\langle r_{\mathfrak{B}}, x^{*} \otimes y^{*}\right\rangle.
	\end{align*}
	Hence, $r+\tau(r)=-\lambda r_{\mathfrak{B}}$ if and only if $P_{r}$ satisfies Eq.~\eqref{eq:qrbp}.
\end{proof}

As a direct consequence, a quadratic Rota-Baxter Poisson algebra of zero weight gives rise to a triangular Poisson bialgebra in the following sense.

\begin{proposition} \label{rbfna2}
	Let $\left((A,[\ ,\ ]_A,\cdot_{A}), \mathfrak{B}, P\right)$ be a quadratic Rota-Baxter Poisson algebra of weight $0$ and $I_\mathfrak{B}: A^* \to A$ be the induced linear isomorphism by $\mathfrak{B}$.
	Then $(A, [\ ,\ ]_A, \cdot_A, \delta_r, \Delta_r)$ is a triangular Poisson bialgebra where $r\in A \otimes A$ is the 2-tensor form of $P \circ I_\mathfrak{B}$ given by
	\begin{equation}
		\langle r, x^* \otimes y^*\rangle := \langle (P\circ I_\mathfrak{B})(x^*), y^*\rangle, \;\; \forall x^*, y^* \in A^*. \label{eq:pnbf}
	\end{equation}
\end{proposition}
\begin{proof}
	It follows from Proposition~\ref{pro:pybe2h} and Lemma~\ref{rbfna0} with $r +\tau (r)=0$.
\end{proof}

The following theorem shows that a factorizable Poisson bialgebra naturally gives rise to a quadratic Rota-Baxter Poisson algebra of nonzero weight.

\begin{theorem}\label{thm:fpb2qrp}
	Let $(A, [\ ,\ ]_A, \cdot_A, \delta_r, \Delta_r)$ be a factorizable Poisson bialgebra with $I_r = r_{+}- r_{-}$.
	Then $((A, [\ ,\ ]_A, \cdot_A), \mathfrak{B}, P)$ is a quadratic Rota-Baxter Poisson algebra of weight $\lambda$, where the linear map $P: A \to A$ and bilinear form $\mathfrak{B} \in \otimes^2 A^*$ are respectively defined by
	\begin{align*}
		P &= -\lambda r_{+} \circ I_r^{-1}, \quad \lambda \neq 0, \\
		\mathfrak{B}(a, b) &= -\lambda\langle I_r^{-1}(a), b \rangle, \;\; \forall a, b \in A.
	\end{align*}
\end{theorem}
\begin{proof}
	Clearly, $\mathfrak{B}$ is a nondegenerate symmetric bilinear form and $I_\mathfrak{B}^{-1} = - \lambda I_r^{-1}$.
	Immediately, we have $I_r = -\lambda I_\mathfrak{B}$ and thus $r + \tau(r) = -\lambda r_\mathfrak{B}$.
	Noting that the symmetric part of $r$ is $(\ad, L)$-invariant, we have $r_\mathfrak{B}$ is $(\ad, L)$-invariant and thus by Proposition~\ref{qparinv}, $(A, [\ ,\ ]_A, \cdot_A, \mathfrak{B})$ is a quadratic Poisson algebra.
	Moreover, we show that $P$ is a Rota-Baxter operator of weight $\lambda$ on the Poisson algebra $(A, [\ ,\ ]_A, \cdot_A)$ by Proposition~\ref{pro:pybe2h} and $P$ satisfies Eq.~\eqref{eq:qrbp} by Lemma~\ref{rbfna0}.
	Thus, $((A, [\ ,\ ]_A, \cdot_A), \mathfrak{B}, P)$ is a quadratic Rota-Baxter Poisson algebra of weight $\lambda$.
\end{proof}

\begin{corollary}
	Let $(A, [\ ,\ ]_A, \cdot_A, \delta_r, \Delta_r)$ be a factorizable Poisson bialgebra with $I_r = r_{+}- r_{-}$, and $P =- \lambda r_{+} \circ I_r^{-1}$ be the induced Rota-Baxter operator of weight $\lambda \neq 0$.
	Let $(A, [\ ,\ ]_P, \cdot_P)$ be the descendent Poisson algebra of the Rota-Baxter algebra $(A, [\ ,\ ]_A, \cdot_A, P)$.
	Then $((A, [\ ,\ ]_P, \cdot_P)$, $(A^*, [\ ,\ ]_{I_r}, \cdot_{I_r}))$ is a Poisson bialgebra, where 
	\begin{align*}
		[x^*, y^*]_{I_r} &:= -\lambda I_r^{-1}([\frac{1}{\lambda}I_r(x^*), \frac{1}{\lambda}I_r(y^*)]_A), \\
		x^* \cdot_{I_r} y^* &:= -\lambda I_r^{-1}((\frac{1}{\lambda}I_r(x^*))\cdot_A (\frac{1}{\lambda}I_r(y^*))), \;\; \forall x^*, y^* \in A^*.
	\end{align*}
	Moreover, $-\frac{1}{\lambda}I_r: A^* \to A$ is a Poisson bialgebra isomorphism from $((A^*, [\ ,\ ]_r, \cdot_r), (A, [\ ,\ ]_A, \cdot_A))$ to $((A$, $[\ ,\ ]_P$, $\cdot_P)$, $(A^*, [\ ,\ ]_{I_r}, \cdot_{I_r}))$, where $[\ ,\ ]_r, \; \cdot_r: A^* \otimes A^* \to A^*$ are respectively defined by Eqs.~\eqref{eq:pcbdr1} and \eqref{eq:pcbdr2}.
\end{corollary}
\begin{proof}
	For all $x^*, y^* \in A^*$, setting $a = I_r(x^*)$ and $b = I_r(y^*)$, we have 
	\begin{align*}
		-\frac{1}{\lambda}I_r([x^*, y^*]_r)   
		&\overset{\;\eqref{eq:pcbdr1}\;}{=}-\frac{1}{\lambda}I_r(-\ad^*(r_{+}(x^*)) y^* +\ad^*(r_{-}(y^*)) x^*)\\
		&\overset{\eqref{eq:adi1}}{=}-\frac{1}{\lambda} ([ r_{+}(x^*), I_r(y^*)]_A - [r_{-}(y^*),  I_r(x^*)]_A )\\
		 & \overset{\hphantom{\eqref{eq:adi1}}}{=} \frac{1}{\lambda^2}([PI_r(x^*), I_r(y^*)]_A + [I_r(x^*), PI_r(y^*)]_A + \lambda [I_r(x^*), I_r(y^*)]_A)\\
		  & \overset{\eqref{eq:rdpal}}{=} [ -\frac{1}{\lambda}I_r(x^*), - \frac{1}{\lambda}I_r(y^*)]_P.
	\end{align*}
	Similarly, we have 
	\begin{equation*}
		-\frac{1}{\lambda}I_r(x^* \cdot_r y^*) = (-\frac{1}{\lambda}I_r(x^*)) \cdot_P  (-\frac{1}{\lambda}I_r(y^*)).
	\end{equation*}
	Hence, $-\frac{1}{\lambda}I_r: (A^*, [\ ,\ ]_r, \cdot_r) \to (A, [\ ,\ ]_P, \cdot_P)$ is a Poisson algebra homomorphism.  
	
	Noting that $\frac{1}{\lambda}I_r^* = \frac{1}{\lambda}I_r$, we have 
	\begin{align*}
		-\frac{1}{\lambda}I^*_r([x^*, y^*]_{I_r}) = [-\frac{1}{\lambda}I_r(x^*), -\frac{1}{\lambda}I_r(y^*)]_A = [-\frac{1}{\lambda}I^*_r(x^*), -\frac{1}{\lambda}I^*_r(y^*)]_A, \\
		-\frac{1}{\lambda}I^*_r(x^*  \cdot_{I_r} y^*) = (-\frac{1}{\lambda}I_r(x^*)) \cdot_A (-\frac{1}{\lambda}I_r(y^*)) = (-\frac{1}{\lambda}I^*_r(x^*)) \cdot_A (-\frac{1}{\lambda}I^*_r(y^*)),
	\end{align*}
	which means that $-\frac{1}{\lambda}I_r^*: (A^*, [\ ,\ ]_{I_r}, \cdot_{I_r}) \to (A, [\ ,\ ]_A, \cdot_A)$ is also a Poisson algebra homomorphism.
	Note that $((A^*, [\ ,\ ]_r, \cdot_r), (A, [\ ,\ ]_A, \cdot_A))$ is a Poisson bialgebra.
	Therefore, by Proposition~\ref{prop:tmpb}, $((A, [\ ,\ ]_P, \cdot_P), (A^*, [\ ,\ ]_{I_r}, \cdot_{I_r}))$ is a Poisson bialgebra.
	Evidently, $-\frac{1}{\lambda}I_r$ is an isomorphism of Poisson bialgebras.
\end{proof}

As a counterpart to Theorem~\ref{thm:fpb2qrp}, the following theorem shows that a quadratic Rota-Baxter Poisson algebra of nonzero weight induces a factorizable Poisson bialgebra, thereby refining the one-to-one correspondence between factorizable Poisson bialgebras and quadratic Rota-Baxter Poisson algebras of nonzero weight.
\begin{theorem}\label{thm:fpb2qrpx}
	Let $((A, [\ ,\ ]_A, \cdot_A), \mathfrak{B}, P)$ be a quadratic Rota-Baxter Poisson algebra of weight $\lambda \neq 0$, and $I_\mathfrak{B}: A^* \to A$ be the induced linear isomorphism by $\mathfrak{B}$. 
	Let $r \in A \otimes A$ be the 2-tensor form of $P \circ I_\mathfrak{B}$ given by Eq.~\eqref{eq:pnbf}. 
	Then $r$ is a solution of the Poisson Yang-Baxter equation in $(A, [\ ,\ ]_A, \cdot_A)$ and gives rise to a factorizable Poisson bialgebra $(A, [\ ,\ ]_A, \cdot_A, \delta_r, \Delta_r)$. 
\end{theorem}
\begin{proof}
	Let $r_\mathfrak{B} \in A \otimes A$ be the 2-tensor form of $I_\mathfrak{B}$ given by Eq.~\eqref{eq:sxrbt2}. 
	By Proposition~\ref{qparinv} and Lemma~\ref{rbfna0}, $r_\mathfrak{B}$ is $(\ad, L)$-invariant and $r + \tau(r) = -\lambda r_\mathfrak{B}$, which show that the symmetric part of $r$ is also $(\ad, L)$-invariant and $I_r = -\lambda I_\mathfrak{B}$ is a linear isomorphism. 
	Moreover, by Proposition~\ref{pro:pybe2h}, $r$ satisfies the Poisson Yang-Baxter equation since $P$ is a Rota-Baxter operator of weight $\lambda$ on $(A, [\ ,\ ]_A, \cdot_A)$. 
	Thus, $(A, [\ ,\ ]_A, \cdot_A, \delta_r, \Delta_r)$ is a factorizable Poisson bialgebra.
\end{proof}

\begin{corollary}
	Let $(A, [\ ,\ ]_A, \cdot_{A})$ be a Poisson algebra, 
	$\{e_1, e_2, \cdots, e_n\}$ be a basis of $A$ and $\{e_1^*, e_2^*, \cdots, e_n^*\}$ be the dual basis of $A^*$.
	Then $r = \sum_{i} e_i^* \otimes e_i$ is a solution of the Poisson Yang-Baxter equation in $(A \ltimes_{-\mathrm{ad}^{*},L^{*}} A^{*}, [\ ,\ ], \cdot)$ and gives rise to a factorizable Poisson bialgebra $(A \ltimes_{-\mathrm{ad}^{*},L^{*}} A^{*}, [\ ,\ ], \cdot, \delta_r, \Delta_r)$. 
\end{corollary}
\begin{proof}
	Clearly, $\id: A \to A$ is a Rota-Baxter operator of weight $-1$ on $(A, [\ ,\ ]_A, \cdot_{A})$.
	Then by Proposition~\ref{exrbfna}, $\left((A \ltimes_{-\mathrm{ad}^{*},L^{*}} A^{*}, [\ ,\ ], \cdot), \mathfrak{B}_d, \id + 0^*\right)$ is a quadratic Rota-Baxter Poisson algebra of weight $-1$, where $\mathfrak{B}_d$ is defined by Eq.~\eqref{abm}.
	Now the conclusion follows from Theorem~\ref{thm:fpb2qrpx}. 
\end{proof}

\begin{proposition}
	Let $(A, [\ ,\ ]_A, \cdot_A, \delta_r, \Delta_r)$ be a factorizable Poisson bialgebra, which corresponds to a quadratic Rota-Baxter Poisson algebra $((A, [\ ,\ ]_A, \cdot_A), \mathfrak{B}, P)$ of weight $\lambda \neq 0$ via Theorems~\ref{thm:fpb2qrp} and \ref{thm:fpb2qrpx}. 
	Then the factorizable Poisson bialgebra $(A, [\ ,\ ]_A, \cdot_A, \delta_{\tau(r)}, \Delta_{\tau(r)})$ corresponds to the quadratic Rota-Baxter Poisson algebra $((A, [\ ,\ ]_A, \cdot_A), \mathfrak{B}, \tilde{P})$ of weight $\lambda \neq 0$ where $\tilde{P} := - \lambda \id - P$. 
	In conclusion, we have the following commutative diagram.
	\begin{displaymath}
		\xymatrix{ 
			(A, [\ ,\ ]_A, \cdot_A, \delta_r, \Delta_r) \ar@{<->}[rr]^{\mathrm{Prop.}~\ref{xtquasipba}}\ar@{->}[d]^{\mathrm{Thm.}~\ref{thm:fpb2qrp}}&  &  (A, [\ ,\ ]_A, \cdot_A,  \delta_{\tau(r)},\Delta_{\tau(r)}) \ar@{->}[d]^{\mathrm{Thm.}~\ref{thm:fpb2qrp}} \\
			((A, [\ ,\ ]_A, \cdot_A), \mathfrak{B}, P) \ar@<1.5ex>[u]^{\mathrm{Thm.}~\ref{thm:fpb2qrpx}} \ar@{<->}[rr]^{\mathrm{Prop}.~\ref{PandtauP} }& &   ((A, [\ ,\ ]_A, \cdot_A), \mathfrak{B}, \tilde{P}) \ar@<1.5ex>[u]^{\mathrm{Thm.}~\ref{thm:fpb2qrpx}}}
	\end{displaymath}
\end{proposition}
\begin{proof}
	By Theorem~\ref{thm:fpb2qrp}, the factorizable Poisson bialgebra $(A, [\ ,\ ]_A, \cdot_A, \delta_{\tau(r)}, \Delta_{\tau(r)})$ induces a quadratic Rota-Baxter Poisson algebra $(A, [\ ,\ ]_A, \cdot_A, \mathfrak{B}^{\prime}, P^{\prime})$ of weight $\lambda \ne 0$, where
	\begin{align*}
		\mathfrak{B}^{\prime}(a, b)&=-\lambda\langle I_{\tau(r)}^{-1}(a), b \rangle=-\lambda\left\langle I_r^{-1}(a), b\right\rangle=\mathfrak{B}(a, b), \quad \forall a, b \in A,\\
		P^{\prime} &= -\lambda (\tau(r))_{+} \circ I_{\tau(r)}^{-1}= \lambda r_{-} \circ I_{r}^{-1} = \lambda r_{+} \circ I_{r}^{-1} - \lambda I_{r} \circ I_{r}^{-1} = -P -\lambda \mathrm{id}_A = \tilde{P}.
	\end{align*}
	Thus, $(A, [\ ,\ ]_A, \cdot_A, \delta_{\tau(r)},\Delta_{\tau(r)})$ induces a quadratic Rota-Baxter Poisson algebra $((A, [\ ,\ ]_A, \cdot_A)$, $\mathfrak{B}, \tilde{P})$ of weight $\lambda$ by Theorem~\ref{thm:fpb2qrp}.
	Conversely, by a similar argument, we show that the quadratic Rota-Baxter Poisson algebra $((A, [\ ,\ ]_A, \cdot_A)$, $\mathfrak{B}, \tilde{P})$ of weight $\lambda$ induces the factorizable Poisson bialgebra $(A, [\ ,\ ]_A, \cdot_A, \delta_{\tau(r)},\Delta_{\tau(r)})$ via Theorem~\ref{thm:fpb2qrpx}.
\end{proof}

%%%%%%%%%%%%%%%%%%%%%%%%%%%%%%%%%%%%%%%%%%%%%%%%%%%%%%%%%%%%%%%%%%%%%%%%%%%%%%%%
%%%%%%%%%%%%%%%%%%%%%%%%%%%%%%%%%%%%%%%%%%%%%%%%%%%%%%%%%%%%%%%%%%%%%%%%%%%%%%%%
%%%%%%%%%%%%%%%%%%%%%%%%%%%%%%%%%%%%%%%%%%%%%%%%%%%%%%%%%%%%%%%%%%%%%%%%%%%%%%%%
\section{Quasi-triangular Poisson bialgebras via quasi-triangular differential ASI bialgebras}\label{sec:qpbfqdasi}
In this section, we generalize the construction of Poisson algebras from commutative algebras with a pair of commuting derivations to the context of quasi-triangular bialgebras.
We establish the quasi-triangular and factorizable theories for differential ASI bialgebras, and then construct quasi-triangular and factorizable Poisson bialgebras from quasi-triangular and factorizable (commutative and cocommutative) differential ASI bialgebras respectively.

%In this section, we always assume that $(A, \cdot_{A})$ is a commutative algebra.

Recall that an \textbf{associative coalgebra} $(A, \Delta)$ is a vector space $A$ with a linear map $\Delta: A \to A \otimes A$ satisfying the coassociative law:
\begin{equation*}
	(\Delta \otimes \id)\Delta = (\id \otimes \Delta) \Delta.
\end{equation*}
An associative coalgebra $(A, \Delta)$ is called \textbf{cocommutative}
if $\Delta = \tau \Delta$.

\begin{definition}\label{def:asi_bialgebra}
	(\cite{bai2010double})
	An \textbf{antisymmetric infinitesimal bialgebra} or simply an \textbf{ASI bialgebra} is a triple $(A, \cdot_{A}, \Delta)$ consisting of a vector space $A$ and linear maps $\cdot_{A}: A \otimes A \to A$ and $\Delta: A \to A \otimes A$ such that
	\begin{enumerate}
		\item $(A, \cdot_{A})$ is an associative algebra;
		\item $(A, \Delta)$ is an associative coalgebra;
		\item the following equations hold for all $a, b \in A$:
		\begin{align*}
			\Delta(a \cdot_{A} b) &= (R(b) \otimes \id )\Delta(a) + (\id \otimes L(a)) \Delta(b),  \\
			(L(a) \otimes \id - \id \otimes R(a))\Delta(b) &= \tau \left( (\id \otimes R(b) - L(b) \otimes \id )\Delta(a) \right).
		\end{align*}
	\end{enumerate}
\end{definition}

\begin{definition}
	Let $(A, \cdot_{A})$ be an associative algebra.
	A linear map $\partial: A \to A$ is called a \textbf{derivation} if the Leibniz rule is satisfied, i.e.,
	\begin{equation*}
		\partial(a \cdot_{A} b) = \partial(a) \cdot_{A} b + a \cdot_{A} \partial(b), \;\; \forall a, b \in A.
	\end{equation*}
	A \textbf{differential algebra} is a triple $(A, \cdot_{A}, \Phi)$, where $(A, \cdot_{A})$ is an associative algebra and $\Phi=\{\partial_{i}: A \to A\}_{i=1}^m$ is a finite set of commuting derivations.
	A differential algebra $(A, \cdot_{A}, \Phi)$ is called \textbf{commutative} if $(A, \cdot_{A})$ is commutative.
\end{definition}

\begin{definition}\label{def:coderivation}
	(\cite{doi1981homological})
	Let $(A, \Delta)$ be an associative coalgebra.
	A linear map $\eth: A \to A$ is called a \textbf{coderivation} on $(A, \Delta)$ if the following equation holds:
	\begin{equation}
		\Delta \eth = (\eth \otimes \id + \id \otimes \eth) \Delta. \label{eq:coderivation}
	\end{equation}
	A \textbf{differential coalgebra} is a triple $(A, \Delta, \Psi)$,
	consisting of an associative coalgebra $(A, \Delta)$ and a finite set of commuting coderivations $\Psi = \{\eth_k: A \to A\}_{k=1}^m$.
	A differential coalgebra $(A, \Delta, \Psi)$ is called \textbf{cocommutative} if $(A, \Delta)$ is cocommutative.
\end{definition}

\begin{definition}\label{def:dasi}
	{\rm (\cite{lin2023differential})}
	A \textbf{differential antisymmetric infinitesimal bialgebra} or simply a \textbf{differential ASI bialgebra} is a quintuple
	$(A, \cdot_{A}, \Delta, \Phi, \Psi)$ satisfying
	\begin{enumerate}
		\item $(A, \cdot_{A}, \Delta)$ is an ASI bialgebra;
		\item $(A, \cdot_{A}, \Phi = \{\partial_k\}_{k=1}^m)$ is a differential algebra;
		\item $(A, \Delta, \Psi = \{\eth_k\}_{k=1}^m)$ is a differential coalgebra;
		\item $(A, \cdot_{A}, \Phi)$ is \textbf{$\Psi$-admissible}, that is, the following equations hold:
		\begin{align}
			\eth_k(a) \cdot_A b &= a \cdot_A \partial_k(b) + \eth_k(a \cdot_A b), \label{eq:qadm1} \\
			a \cdot_A \eth_k(b) &= \partial_k(a) \cdot_A b + \eth_k(a \cdot_A b), \;\;
			\forall a,b \in A, \;\; \forall  k=1, \cdots, m. \label{eq:qadm2}
		\end{align}
		\item $(A, \Delta^{*}, \Psi^{*})$ is \textbf{$\Phi^{*}$-admissible}, that is, the following equations hold:
		\begin{align}
			(\partial_k \otimes \id) \Delta &= (\id \otimes \eth_k) \Delta + \Delta \partial_k, \label{eq:psadm1} \\
			(\id \otimes \partial_k) \Delta &= (\eth_k \otimes \id) \Delta + \Delta \partial_k, \;\; \forall  k=1, \cdots, m. \label{eq:psadm2}
		\end{align}
	\end{enumerate}
	A differential ASI bialgebra $(A, \cdot_{A}$, $\Delta$, $\Phi, \Psi)$ is called \textbf{commutative and cocommutative} if $(A$, $\cdot_{A}$, $\Phi)$ is commutative and $(A, \Delta, \Psi)$ is cocommutative.
\end{definition}

\begin{definition}
	Let $(A, \cdot_{A}, \Phi = \{\partial_k\}_{k=1}^m)$ be a differential algebra.
	Suppose that $r \in A \otimes A$ and $\Psi = \{\eth_k: A \to A\}_{k=1}^m$ is a set of commuting linear maps.
	Then $\bm{A}(r) = 0$ with conditions given by the following equations
	\begin{align}
		(\partial_k \otimes \id - \id \otimes \eth_k)(r) &= 0, \label{eq:pqadm1} \\
		(\eth_k \otimes \id - \id \otimes \partial_k)(r) &= 0, \;\; \forall i=1,2, \cdots, m, \label{eq:pqadm2} 
	\end{align}
	is called the \textbf{$\Psi$-admissible associative Yang-Baxter equation} in $(A, \cdot_{A}, \Phi)$ or simply \textbf{$\Psi$-admissible AYBE} in $(A, \cdot_{A}, \Phi)$.
\end{definition}

\begin{lemma}\label{lem:aof}
	Let $(A, \cdot_{A}, \Phi = \{\partial_k\}_{k=1}^m)$ be a differential algebra.
	Suppose that $r \in A \otimes A$ and $\Psi = \{\eth_k: A \to A\}_{k=1}^m$ is a set of commuting linear maps.
	Then Eqs.~\eqref{eq:pqadm1}-\eqref{eq:pqadm2} hold if and only if
	\begin{equation*}
		\partial_{i}r_{+} = r_{+}\eth_{i}^{*}, \;\;
		\partial_{i}r_{-} = r_{-}\eth_{i}^{*}, \;\; \forall i = 1,2, \cdots, m.
	\end{equation*}
\end{lemma}
\begin{proof}
	For all $x^{*}, y^*\in A^{*}$, we have 
	\begin{align*}
		\langle \partial_{i}r_{+}(x^{*}), y^{*}\rangle &= \langle r, x^* \otimes \partial_{i}^{*}(y^{*})\rangle = \langle (\id \otimes \partial_{i})(r), x^{*} \otimes y^{*}\rangle, \\
		\langle r_{+}(\eth_{i}^{*}(x^{*})), y^{*}\rangle &= \langle r, \eth_{i}^{*}(x^*) \otimes y^{*} \rangle = \langle (\eth_{i} \otimes \id)(r), x^{*} \otimes y^{*}\rangle, \\
		\langle \partial_{i}r_{-}(x^{*}), y^{*}\rangle &= -\langle r, \partial_{i}^{*}(y^{*}) \otimes x^{*}  \rangle = -\langle (\partial_{i} \otimes \id)(r), y^{*} \otimes x^{*}\rangle, \\
		\langle r_{-}(\eth_{i}^{*}(x^{*})), y^{*}\rangle &= -\langle r, y^{*} \otimes \eth_{i}^{*}(x^{*}) \rangle = -\langle (\id \otimes \eth_{i})(r), y^{*} \otimes x^{*}\rangle,
	\end{align*}
	which completes the proof.
\end{proof}

\begin{lemma}\label{cdasibi}
	Let $(A, \cdot_{A}, \Phi)$ be a $\Psi$-admissible differential algebra and $r \in A \otimes A$.
	If $r \in A \otimes A$ is a solution of the $\Psi$-admissible AYBE in $(A, \cdot_{A}, \Phi)$ and the symmetric part of $r$ satisfies Eq.~\eqref{eq:liv},
	then $(A, \cdot_{A}, \Delta, \Phi, \Psi)$ is a differential ASI bialgebra, 
	where $\Delta: A \to A \otimes A$ is defined by 
	\begin{equation}
		\Delta(a) = (\id \otimes L(a) - R(a) \otimes \id)(r), \;\; \forall a \in A. \label{eq:dcbd}
	\end{equation}
\end{lemma}
\begin{proof}
	It follows from \cite[Corollary~4.4]{lin2023differential}. 
\end{proof}

\begin{definition}
	Let $(A, \cdot_A, \Phi)$ be a differential algebra and $r \in A \otimes A$.
	Then $r$ is called \textbf{$L$-invariant} if the following equation holds:
	\begin{equation}
		(L(a) \otimes \id - \id \otimes R(a))(r) = 0, \;\; \forall a \in A. \label{eq:livnc}
	\end{equation}
\end{definition}

Recall that a differential ASI bialgebra $(A, \cdot_{A}, \Delta, \Phi, \Psi)$ is called \textbf{coboundary} if $\Delta$ is defined by Eq.~\eqref{eq:dcbd} for some $r \in A \otimes A$.
A coboundary differential ASI bialgebra is denoted by $(A, \cdot_{A}, \Delta_{r}, \Phi, \Psi)$.

\begin{definition}\label{quasidab}
	Let $(A, \cdot_{A}, \Phi)$ be a $\Psi$-admissible differential algebra. 
	If $r$ is a solution of the $\Psi$-admissible AYBE in $(A, \cdot_{A}, \Phi)$ and the symmetric part of $r \in A \otimes  A$ is $L$-invariant, then the coboundary differential ASI bialgebra $(A, \cdot_{A}, \Delta_{r}, \Phi, \Psi)$ induced by $r$ is called a \textbf{quasi-triangular differential ASI bialgebra}.
	In particular, if $r$ is antisymmetric, then $(A, \cdot_{A}, \Delta_{r}, \Phi, \Psi)$ is called a \textbf{triangular differential ASI bialgebra}.
	A quasi-triangular differential ASI bialgebra $(A, \cdot_{A}$, $\Delta_{r}$, $\Phi$, $\Psi)$ is called \textbf{factorizable} if the symmetric part of $r$ is nondegenerate.
\end{definition}

The following proposition justifies the terminology of a factorizable differential ASI bialgebra.
\begin{definition}
	Let $(A, \cdot_{A}, \Phi_{A}=\{\partial_{A, i}\}_{i=1}^m)$ and $(B, \cdot_{B}, \Phi_{B}=\{\partial_{B, i}\}_{i=1}^m)$ be two differential algebras.
	A linear map $\varphi: (A, \cdot_{A}, \Phi_{A}) \to (B, \cdot_{B}, \Phi_{B})$ is called a \textbf{homomorphism} of differential algebras if $\varphi: (A, \cdot_{A}, \Phi_{A}) \to (B, \cdot_{B}, \Phi_{B})$ is a homomorphism of associative algebras satisfying 
	\begin{equation*}
		\varphi \circ \partial_{A, i} = \partial_{B, i} \circ \varphi, \;\; \forall i = 1, 2, \cdots, m.
	\end{equation*}
	If in addition $\varphi$ is a linear isomorphism, then $\varphi$ is called an \textbf{isomorphism} of differential algebras.
\end{definition}

\begin{proposition}\label{prop:fdt}
	Let $(A, \cdot_{A}, \Delta_{r}, \Phi, \Psi)$ be a factorizable differential ASI bialgebra.
	Then $\mathrm{Im}(r_{+} \oplus r_{-})$ is a differential subalgebra of the direct sum differential algebra $A \oplus A$, 
	which is isomorphic to the differential algebra $(A^*, \cdot_r, \Psi^{*})$, where $\cdot_r: A^* \otimes A^* \to A^*$ is defined by 
	\begin{equation}
		x^* \cdot_r y^* = R^*(r_{+}(x^*)) y^* + L^*(r_{-}(y^*)) x^*, \;\; \forall x^*,y^* \in A^*. \label{eq:dcbdr}
	\end{equation}
	Moreover, every $a \in A$ has a unique decomposition 
	$a = a_{+} - a_{-}$ with $(a_{+}, a_{-}) \in \mathrm{Im}(r_{+} \oplus r_{-})$.
\end{proposition}
\begin{proof}
	By \cite[Proposition~3.2]{sheng2023quasi}, we have that both $r_{+}, r_{-}: (A^{*}, \cdot_{r}) \to (A, \cdot_{A})$ are associative algebra homomorphisms, $\mathrm{Im}(r_{+} \oplus r_{-})$ is an associative subalgebra of the direct sum associative algebra $A \oplus A$, which is isomorphic to the associative algebra $(A^*, \cdot_r)$, and every $a \in A$ has a unique decomposition 
	$a = a_{+} - a_{-}$ with $(a_{+}, a_{-}) \in \mathrm{Im}(r_{+} \oplus r_{-})$.
	Furthermore, by Lemma~\ref{lem:aof}, we show that both $r_{+}, r_{-}: (A^{*}, \cdot_{r}, \Psi^{*}) \to (A, \cdot_{A}, \Phi)$ are differential algebra homomorphisms.
	Therefore, $\mathrm{Im}(r_{+} \oplus r_{-})$ is a differential subalgebra of the direct sum differential algebra $A \oplus A$, 
	which is isomorphic to the differential algebra $(A^*, \cdot_r, \Psi^{*})$.
\end{proof}

Let $(A, \cdot_{A}, \Delta, \Phi, \Psi)$ be an arbitrary differential ASI bialgebra and $(\mathfrak{A}, \cdot_{\mathfrak{A}})$ be the associative algebra structure on $A \oplus A^*$ obtained from the matched pair of associative algebras $(A$, $A^*$, $R^*_{\cdot_{A}}$, $L^*_{\cdot_{A}}$, $R^*_{\cdot_{A^*}}$, $L^*_{\cdot_{A^*}})$ (\cite{lin2023differential}).
Moreover, let $\{e_1, e_2, \cdots, e_n\}$ be a basis of $A$ and $\{e_1^*, e_2^*, \cdots, e_n^*\}$ be the dual basis of $A^*$.
Define 
\begin{equation*}
	r = \sum_{i}e_i \otimes e_i^* \in A \otimes A^* \subset \mathfrak{A} \otimes \mathfrak{A}
\end{equation*}
and 
\begin{equation*}
	\Delta_{r}(u) = (\id \otimes L_{\cdot_{\mathfrak{A}}}(u) - R_{\cdot_{\mathfrak{A}}}(u) \otimes \id)(r), \;\; \forall u \in \mathfrak{A},
\end{equation*}
then $(\mathfrak{A}, \cdot_{\mathfrak{A}}, \Delta_{r}, \Phi + \Psi^*, \Psi + \Phi^*)$ is a coboundary differential ASI bialgebra (see \cite[Theorem~4.5]{lin2023differential} for more details).
\begin{theorem}\label{thm:adfd}
	The differential ASI bialgebra $(\mathfrak{A}, \cdot_{\mathfrak{A}}, \Delta_{r}, \Phi + \Psi^*, \Psi + \Phi^*)$ is a quasi-triangular differential ASI bialgebra. 
	Furthermore, it is factorizable.
\end{theorem}
\begin{proof}
	The proof of \cite[Theorem~4.5]{lin2023differential} implies that $r$ is a solution of the $(\Psi + \Phi^*)$-admissible AYBE in $(\mathfrak{A}, \cdot_{\mathfrak{A}}, \Phi + \Psi^*)$.
	By suitable modification to the proof \cite[Theorem~4.5]{lin2023differential}, we show that the symmetric part of $r$ is $L$-invariant.
	Moreover, by the proof of Theorem~\ref{thm:spfb}, we have $I_{r}$ is an isomorphism of vector spaces.
	Hence, the differential ASI bialgebra $(\mathfrak{A}, \cdot_{\mathfrak{A}}, \Delta_{r}, \Phi + \Psi^*, \Psi + \Phi^*)$ is factorizable.
\end{proof}

\begin{definition}
	A bilinear form $\mathfrak{B}(\ ,\ )$ on an associative algebra $(A, \cdot_{A})$ is called \textbf{invariant} if
	\begin{equation*}
		\mathfrak{B}(a \cdot_{A} b, c) = \mathfrak{B}(a, b \cdot_{A} c), \;\;
		\forall a, b, c \in A.
	\end{equation*}
	A \textbf{Frobenius algebra} $(A, \cdot_{A}, \mathfrak{B})$ is an associative algebra $(A, \cdot_{A})$ with a nondegenerate invariant bilinear form $\mathfrak{B}(\ , \ )$.
	A Frobenius algebra $(A, \cdot_{A}, \mathfrak{B})$ is called \textbf{symmetric} if $\mathfrak{B}(\ , \ )$ is symmetric.
\end{definition}

\begin{definition}
	A \textbf{differential Frobenius algebra} is a quadruple $(A, \cdot_{A}, \Phi, \mathfrak{B})$, where $(A, \cdot_{A}$, $\Phi = \{\partial_k\}_{k=1}^m)$ is a differential algebra
	and $(A, \cdot_{A}, \mathfrak{B})$ is a Frobenius algebra. It is
	called {\bf symmetric} if $\mathfrak B$ is symmetric.
	For all $k=1,\cdots, m$, let $\hat{\partial_k}$ be the adjoint linear operator of $\partial_k$ under the nondegenerate bilinear form $\mathfrak{B}$:
	\begin{equation*}
		\mathfrak{B}(\partial_k(a), b) = \mathfrak{B}(a, \hat{\partial_k}(b)), \;\;
		\forall a, b \in A.
	\end{equation*}
	We call $\hat{\Phi} := \{\hat{\partial_k}\}_{k=1}^m$ the \textbf{adjoint of $\Phi = \{\partial_k\}_{k=1}^m$ with respect to $\mathfrak B$}.
\end{definition}
Note that $\hat{\Phi}$ is admissible to $(A, \cdot_{A}, \Phi)$ (\cite[Proposition~3.3]{lin2023differential}).

\begin{definition}
	The triple $((A, \cdot_A, \Phi = \{\partial_k\}_{k=1}^m), \mathfrak{B}, P)$ is called a \textbf{symmetric Rota-Baxter differential Frobenius algebra of weight $\lambda$} if $(A, \cdot_A, \Phi, \mathfrak{B})$ is a symmetric differential Frobenius algebra and $(A, \cdot_A, P)$ is a Rota-Baxter algebra of weight $\lambda$ satisfying the compatibility condition given by Eq.~\eqref{eq:qrbp} such that 
	\begin{equation*}
		\partial_{i} P = P \partial_{i}, \;\; \forall i=1, 2, \cdots, m.
	\end{equation*}
\end{definition}

\begin{lemma}\label{lem:dfof}
	Let $((A, \cdot_{A}, \Phi = \{\partial_k\}_{k=1}^m), \mathfrak{B}, P)$ be a symmetric Rota-Baxter differential Frobenius algebra of weight $\lambda$, and $I_\mathfrak{B}: A^* \to A$ be the induced linear isomorphism by $\mathfrak{B}$. 
	Let $r \in A \otimes A$ be the 2-tensor form of $P \circ I_\mathfrak{B}$ given by Eq.~\eqref{eq:pnbf}
	and $\hat{\Phi} = \{\hat{\partial_k}\}_{k=1}^m$ be the adjoint of $\Phi$ with respect to $\mathfrak B$. 
	Then 
	\begin{equation*}
		\partial_{i}I_{\mathfrak{B}} = I_{\mathfrak{B}}\hat{\partial_{i}}^{*}, \;\;
		\partial_{i}r_{+} = r_{+}\hat{\partial_{i}}^{*}, \;\; \forall i = 1, 2, \cdots, m.
	\end{equation*}
\end{lemma}
\begin{proof}
	For all $a \in A$ and $x^{*} \in A^{*}$, we have 
	\begin{equation*}
		\mathfrak{B}(\partial_{i}I_{\mathfrak{B}}(x^{*}), a) = \mathfrak{B}(I_{\mathfrak{B}}(x^{*}), \hat{\partial_{i}}(a)) = \langle x^*, \hat{\partial_{i}}(a)\rangle = \langle \hat{\partial_{i}}^{*}(x^*), a\rangle = \mathfrak{B}(I_{\mathfrak{B}}\hat{\partial_{i}}^{*}(x^*), a),
	\end{equation*}
	that is, $\partial_{i}I_{\mathfrak{B}} = I_{\mathfrak{B}}\hat{\partial_{i}}^{*}$.
	Therefore,
	\begin{equation*}
		\partial_{i}r_{+} = \partial_{i}P I_{\mathfrak{B}} = P \partial_{i} I_{\mathfrak{B}} = PI_{\mathfrak{B}}\hat{\partial_{i}}^{*} = r_{+}\hat{\partial_{i}}^{*}.
	\end{equation*}
	The proof is completed.
\end{proof}

\begin{proposition}\label{rbfw0rd}
	Let $\left((A, \cdot_A, \Phi), \mathfrak{B}, P\right)$ be a symmetric Rota-Baxter differential Frobenius algebra of weight $0$ and $I_\mathfrak{B}: A^* \to A$ be the induced linear isomorphism by $\mathfrak{B}$.
	Then $(A, \cdot_A, \Delta_r, \Phi, \hat{\Phi})$ is a triangular differential ASI bialgebra where $r\in A \otimes A$ is the 2-tensor form of $P \circ I_\mathfrak{B}$ given by Eq.~\eqref{eq:pnbf}  and $\hat{\Phi}$ is the adjoint of $\Phi$ with respect to $\mathfrak B$.
\end{proposition}
\begin{proof}
	It follows from Lemma~\ref{rbfna0} that $r +\tau (r)=0$.
	Similarly to Proposition~\ref{pro:pybe2h}, we show that $\bm{A}(r) = 0$.
	On the other hand, setting $\Phi = \{\partial_k\}_{k=1}^m$, by Lemma~\ref{lem:dfof}, we have $\partial_{i}r_{+} = r_{+}\hat{\partial_{i}}^{*}$.
	Therefore, by Lemma~\ref{lem:aof}, we show that $r$ is a solution of the $\hat{\Phi}$-admissible AYBE in $(A, \cdot_{A}, \Phi)$.
	Noting that $\hat{\Phi}$ is admissible to $(A, \cdot_{A}, \Phi)$, we complete the proof. 
\end{proof}

\begin{theorem}\label{thm:fdb2qrp}
	Let $(A, \cdot_A, \Delta_r, \Phi = \{\partial_k\}_{k=1}^m, \Psi = \{\eth_k\}_{k=1}^m)$ be a factorizable differential ASI bialgebra with $I_r = r_{+}- r_{-}$.
	Then $((A, \cdot_A, \Phi), \mathfrak{B}, P)$ is a symmetric Rota-Baxter differential Frobenius algebra of weight $\lambda$ such that the adjoint of $\Phi$ with respect to $\mathfrak{B}$ is $\Psi$, where the linear map $P: A \to A$ and bilinear form $\mathfrak{B} \in \otimes^2 A^*$ are respectively defined by
	\begin{align*}
		P &= -\lambda r_{+} \circ I_r^{-1}, \quad \lambda \neq 0, \\
		\mathfrak{B}(a, b) &= -\lambda\langle I_r^{-1}(a), b \rangle, \;\; \forall a, b \in A.
	\end{align*}
\end{theorem}
\begin{proof}
	It follows from \cite[Corollary~4.8]{sheng2023quasi} that $(A, \cdot_{A}, \mathfrak{B})$ is a symmetric Frobenius algebra and $(A, \cdot_A, P)$ is a Rota-Baxter algebra of weight $\lambda$ such that Eq.~\eqref{eq:qrbp} holds.
	For all $a, b\in A$, we have
	\begin{align*}
		\mathfrak{B}(\partial_{i}(a), b) = -\lambda \langle I_{r}^{-1}(\partial_{i}(a)), b\rangle = -\lambda \langle I_{r}^{-1}\partial_{i}I_{r}I_{r}^{-1}(a), b\rangle, \\
		\mathfrak{B}(a, \eth_{i}(b)) = -\lambda \langle I_{r}^{-1}(a), \eth_{i}(b)\rangle = -\lambda \langle I_{r}^{-1}I_{r}\eth_{i}^{*}I_{r}^{-1}(a), b\rangle,
	\end{align*}
	and by Lemma~\ref{lem:aof}, we have 
	\begin{equation*}
		\partial_{i}I_{r} = \partial_{i}(r_{+}-r_{-}) = (r_{+}-r_{-})\eth_{i}^{*} = I_{r}\eth_{i}^{*}.
	\end{equation*}
	Therefore, $\mathfrak{B}(\partial_{i}(a), b) = \mathfrak{B}(a, \eth_{i}(b))$.
	Furthermore, we have 
	\begin{equation*}
		\partial_{i} P = - \lambda \partial_{i} r_{+}  I_{r}^{-1} =  - \lambda r_{+} \eth_{i}^{*} I_{r}^{-1} = - \lambda r_{+} I_{r}^{-1} \partial_{i} = P \partial_{i}.
	\end{equation*}
	The proof is completed.
\end{proof}

\begin{theorem}\label{thm:fdb2qrpx}
	Let $((A, \cdot_{A}, \Phi = \{\partial_k\}_{k=1}^m), \mathfrak{B}, P)$ be a symmetric Rota-Baxter differential Frobenius algebra of weight $\lambda \neq 0$, and $I_\mathfrak{B}: A^* \to A$ be the induced linear isomorphism by $\mathfrak{B}$. 
	Let $r \in A \otimes A$ be the 2-tensor form of $P \circ I_\mathfrak{B}$ given by  Eq.~\eqref{eq:pnbf} and $\hat{\Phi}$ be the adjoint of $\Phi$ with respect to $\mathfrak B$. 
	Then $r$ is a solution of the $\hat{\Phi}$-admissible AYBE in $(A, \cdot_A, \Phi)$ and gives rise to a factorizable differential ASI bialgebra $(A, \cdot_{A}, \Delta_r, \Phi, \hat{\Phi})$. 
\end{theorem}
\begin{proof}
	By Lemma~\ref{rbfna0},  $r + \tau(r)$   is equal to $-\lambda r_\mathfrak{B}$ and $I_{r} = -\lambda I_{\mathfrak{B}}$.
	Then by Lemma~\ref{lem:dfof}, we have 
	\begin{equation*}
		\partial_{i}r_{+} = r_{+}\hat{\partial_{i}}^{*} \text{\;\; and \;\;}
		\partial_{i}r_{-} = \partial_{i}(r_{+} - I_{r}) = r_{+}\hat{\partial_{i}}^{*} - I_{r} \hat{\partial_{i}}^{*} = r_{-}\hat{\partial_{i}}^{*}.
	\end{equation*}
	Moreover, similarly to Theorem~\ref{thm:fpb2qrpx}, we show that $\bm{A}(r) = 0$ and the symmetric part of $r$ is $L$-invariant.
	Hence, by Lemma~\ref{lem:aof}, $r$ is a solution of the $\hat{\Phi}$-admissible AYBE in $(A, \cdot_{A}, \Phi)$, which completes the proof.
\end{proof}

 We now turn to the constructions of quasi-triangular and factorizable Poisson bialgebras from quasi-triangular and factorizable (commutative and cocommutative) differential ASI bialgebras.
The following result is known (cf. \cite{bhaskara1988poisson}).
\begin{proposition}\label{prop:ipfd}
	Let $(A, \cdot_{A}, \Phi = \{\partial_{1}, \partial_{2}\})$ be a commutative differential algebra.
	Then $(A$, $[\ ,\ ]_{A}$, $\cdot_{A})$ is a Poisson algebra, called the \textbf{induced Poisson algebra of $(A, \cdot_{A}, \Phi)$}, where $[\ , \ ]: A \otimes A \to A$ is defined by 
	\begin{equation}
		[a, b]_{A} := \partial_{1}(a)\cdot_{A} \partial_{2}(b) - \partial_{2}(a)\cdot_{A} \partial_{1}(b), \;\; \forall a, b \in A.
	\end{equation}
\end{proposition}

\begin{lemma}\label{lem:aybe2pybe}
	{\rm (\cite{lin2023differential})}
	Let $(A, \cdot_{A}, \Phi = \{\partial_1, \partial_2\})$ be a $\Psi = \{\eth_1, \eth_2\}$-admissible commutative differential algebra and $(A, [\ ,\ ]_{A}, \cdot_{A})$ be the induced Poisson algebra of $(A, \cdot_{A}, \Phi)$.
	Suppose that the following equation holds
	\begin{equation}
		\eth_2(\partial_1(a)) \cdot b = \eth_1(\partial_2(a)) \cdot b , \;\;
		\forall a, b \in A. \label{eq:vip}
	\end{equation}
	Then every solution of the $\Psi$-admissible AYBE in $(A, \cdot_{A}, \Phi)$ is a solution of the Poisson Yang-Baxter equation in the Poisson algebra $(A, [\ ,\ ]_{A}, \cdot_{A})$.
\end{lemma}

\begin{theorem}\label{thm:dasi2pb}
	If $(A, \cdot_{A}, \Delta_{r}, \Phi, \Psi)$ is a quasi-triangular (resp. triangular, factorizable) commutative and cocommutative differential ASI bialgebra, and Eq.~\eqref{eq:vip} holds,
	then $(A, [\ ,\ ]_{A}, \cdot_{A}, \delta_r, \Delta_{r})$ is a quasi-triangular (resp. triangular, factorizable) Poisson bialgebra through $r$, where $(A, [\ ,\ ]_{A}, \cdot_{A})$ is the induced Poisson algebra of $(A, \cdot_{A}, \Phi)$ and 
	$\delta_r: A \to A \otimes A$ is defined by
	\begin{equation}
		\delta_r = (\eth_1 \otimes \eth_2 - \eth_2 \otimes \eth_1)
		\Delta_r. \label{eq:coP}
	\end{equation}
\end{theorem}
\begin{proof}
	For all $a \in A$, we have
	\begin{align*}
		&(\ad(a) \otimes \id + \id \otimes \ad(a))(r+\tau(r)) \\
		&\overset{\hphantom{\eqref{eq:qadm1}}}{=} \left((L(\partial_{1}(a))\partial_{2} - L(\partial_{2}(a))\partial_{1} ) \otimes \id + \id \otimes (L(\partial_{1}(a))\partial_{2} - L(\partial_{2}(a))\partial_{1})\right)(r+\tau(r)) \\
		&\overset{\eqref{eq:qadm1}}{=} \left( L(\partial_{1}(a)) \otimes \eth_{2} - L(\partial_{2}(a)) \otimes \eth_{1}  - \id \otimes \eth_{2} L(\partial_{1}(a))  + \id \otimes \eth_{1} L(\partial_{2}(a)) \right)(r+\tau(r)) \\
		&\hspace{0.6cm} + (\id \otimes L(\eth_{2}\partial_{1}(a)) - \id \otimes L(\eth_{1}\partial_{2}(a)))(r+\tau(r)) \\
		&\overset{\eqref{eq:vip}}{=} \left( (\id \otimes \eth_{2})(L(\partial_{1}(a)) \otimes \id - \id \otimes L(\partial_{1}(a))) + (\id \otimes \eth_{1})(L(\partial_{2}(a)) \otimes \id - \id \otimes L(\partial_{2}(a))) \right)(r + \tau(r))\\
		&\hspace{0.1cm}= 0.
	\end{align*}
	Hence, the symmetric part of $r$ is $(\ad, L)$-invariant.
	By Lemma~\ref{lem:aybe2pybe}, $r$ is a solution of the Poisson Yang-Baxter equation in $(A, [\ ,\ ]_{A}, \cdot_{A})$.
	Furthermore, 
	\begin{align*}
		\delta_r(a)&\overset{\hphantom{\eqref{eq:pqadm1}}}{=} (\eth_1 \otimes \eth_2 - \eth_2 \otimes \eth_1) \Delta_r(a) =(\eth_1 \otimes \eth_2 - \eth_2 \otimes \eth_1) (\id \otimes L(a) - L(a) \otimes \id )(r) \\
		&\overset{\eqref{eq:pqadm1}}{=} (\id \otimes \eth_2 L(a) \partial_1 - \id \otimes \eth_1 L(a) \partial_2 )(r) - (\eth_1 L(a) \partial_2 \otimes \id - \eth_2 L(a) \partial_1 \otimes \id)(r) \\
		&\overset{\eqref{eq:qadm1}}{=} ( \ad(a) \otimes \id + \id \otimes \ad (a))(r).
	\end{align*}
	It is now obvious that the theorem holds.
\end{proof}

Let $(A, \cdot_{A}, \Delta_{r}, \Phi=\{\partial_{1}, \partial_{2}\}, \Psi=\{\eth_{1}, \eth_{2}\})$ be a communicative and cocommutative factorizable differential ASI bialgebra.
By Proposition~\ref{prop:fdt}, let $(\mathrm{Im}(r_{+} \oplus r_{-}), \cdot, \Phi + \Phi = \{\partial_{1} + \partial_{1}, \partial_{2} + \partial_{2}\})$ be the differential subalgebra of the direct sum differential algebra $A \oplus A$, which is isomorphic to the differential algebra $(A^*, \cdot_r, \Psi^{*})$, where $\cdot_r: A^* \otimes A^* \to A^*$ is defined by \eqref{eq:pcbdr2}.
On the other hand, suppose that Eq.~\eqref{eq:vip} holds. 
Let $(A, [\ ,\ ]_{A}, \cdot_{A}, \delta_{r}, \Delta_{r})$ be the induced factorizable Poisson bialgebra in Theorem~\ref{thm:dasi2pb}.
Then by Proposition~\ref{prop:fpt}, let $(\mathrm{Im}(r_{+} \oplus r_{-}), [\ ,\ ], \cdot)$ be the Poisson subalgebra of the direct sum Poisson algebra $A \oplus A$, 
which is isomorphic to the Poisson algebra $(A^*, [\ ,\ ]_r, \cdot_r)$, where $[\ ,\ ]_r, \cdot_r: A^* \otimes A^* \to A^*$ are respectively defined by Eqs.~\eqref{eq:pcbdr1} and \eqref{eq:pcbdr2}.
\begin{corollary}
	With the conditions as above. 
	The induced Poisson algebras of $(\mathrm{Im}(r_{+} \oplus r_{-}), \cdot, \Phi + \Phi)$ and $(A^*, \cdot_r, \Psi^{*})$ are exactly $(\mathrm{Im}(r_{+} \oplus r_{-}), [\ ,\ ], \cdot)$ and $(A^*, [\ ,\ ]_r, \cdot_r)$ respectively.  
	In conclusion, we have the following commutative diagram
	\begin{displaymath}
		\xymatrix@C=2cm{ 
			(A, \cdot_A, \Delta_r, \Phi, \Psi) \ar@{->}[d]^{\mathrm{Thm.}~\ref{thm:dasi2pb}} \ar@{->}[r]^{\mathrm{Prop.}~\ref{prop:fdt}} & (\mathrm{Im}(r_{+} \oplus r_{-}), \cdot, \Phi + \Phi) \ar@{->}[r]^-{\simeq} \ar@{->}[d]^{\mathrm{Prop.}~\ref{prop:ipfd}} & (A^*, \cdot_r, \Psi^{*}) \ar@{->}[d]^{\mathrm{Prop.}~\ref{prop:ipfd}} \\
			(A, [\ ,\ ]_A, \cdot_A,  \delta_{r},\Delta_{r}) \ar@{->}[r]^{\mathrm{Prop.}~\ref{prop:fpt}} & (\mathrm{Im}(r_{+} \oplus r_{-}), [\ ,\ ], \cdot) \ar@{->}[r]^-{\simeq} & (A^*, [\ ,\ ]_r, \cdot_r) 
		}
	\end{displaymath}
\end{corollary}
\begin{proof}
	By Theorem~\ref{thm:dasi2pb},
	the induced Poisson algebra of $(A^*, \cdot_r, \Psi^{*})$ is exactly $(A^*, [\ ,\ ]_r, \cdot_r)$.
	Note that $(A, [\ ,\ ]_{A}, \cdot_{A})$ is the induced Poisson algebra of $(A, \cdot_{A}, \Phi)$, it is straightforward that the induced Poisson algebra of $(\mathrm{Im}(r_{+} \oplus r_{-}), \cdot, \Phi + \Phi)$ is exactly $(\mathrm{Im}(r_{+} \oplus r_{-}), [\ ,\ ], \cdot)$.
	The proof is completed.
\end{proof}

Let $(A, \cdot_{A}, \Delta, \Phi=\{\partial_{1}, \partial_{2}\}, \Psi=\{\eth_{1}, \eth_{2}\})$ be a communicative and cocommutative differential ASI bialgebra.
Let $(A, [\ ,\ ]_{A}, \cdot_{A})$ be the induced Poisson algebra of $(A, \cdot_{A}, \Phi)$. Suppose that Eq.~\eqref{eq:vip} and the following equation hold:
\begin{equation}
	(\eth_2\partial_1 \otimes \id)\Delta = (\eth_1\partial_2 \otimes \id)\Delta. \label{eq:vip1}
\end{equation}
Then $(A, [\ ,\ ]_{A}, \cdot_{A}, \delta, \Delta)$ is a Poisson bialgebra where $\delta: A \to A \otimes A$ is defined by $\delta = (\eth_1 \otimes \eth_2 - \eth_2 \otimes \eth_1)
\Delta$ (\cite[Theorem~5.24]{lin2023differential}).
Let $(\mathfrak{A}, \cdot_{\mathfrak{A}}, \Delta_{r}, \Phi + \Psi^*, \Psi + \Phi^*)$ be the factorizable differential ASI bialgebra obtained in Theorem~\ref{thm:adfd} arising from $(A$, $\cdot_{A}$, $\Delta$, $\Phi$, $\Psi)$ and $(\mathfrak{A}$, $[\ ,\ ]_\mathfrak{A}$, $\cdot_\mathfrak{A}$, $\delta_r$, $\Delta_r)$ be the factorizable Poisson bialgebra obtained in Theorem~\ref{thm:spfb} arising from $(A, [\ ,\ ]_{A}, \cdot_{A}, \delta, \Delta)$.
Noting that the Poisson algebra structure $(\mathfrak{A}, [\ ,\ ]_\mathfrak{A}, \cdot_\mathfrak{A})$ is exactly the induced Poisson algebra of $(\mathfrak{A}, \cdot_{\mathfrak{A}}, \Phi + \Psi^*)$ (\cite[Remark~5.25]{lin2023differential}) and Eq.~\eqref{eq:vip} also holds for the $(\Psi+\Phi^*)$-admissible commutative differential algebra  $(\mathfrak{A}, \cdot_{\mathfrak{A}}, \Phi+\Psi^{*})$, we have the following conclusion.
\begin{corollary}
	With the conditions as above.
	$(\mathfrak{A}$, $[\ ,\ ]_\mathfrak{A}$, $\cdot_\mathfrak{A}$, $\delta_r$, $\Delta_r)$ is the induced factorizable Poisson bialgebra of the factorizable differential ASI bialgebra $(\mathfrak{A}, \cdot_{\mathfrak{A}}, \Delta_{r}, \Phi + \Psi^*, \Psi + \Phi^*)$.
	In conclusion, we have the following commutative diagram
	\begin{displaymath}
		\xymatrix@C=3cm{ 
			(A, \cdot_A, \Delta, \Phi, \Psi)  \ar@{->}[r]^-{\mathrm{Thm.}~\ref{thm:adfd}} \ar@{->}[d]& (\mathfrak{A}, \cdot_{\mathfrak{A}}, \Delta_{r}, \Phi + \Psi^*, \Psi + \Phi^*) \ar@{->}[d]^-{\mathrm{Thm.}~\ref{thm:dasi2pb}} \\
			(A, [\ ,\ ]_A, \cdot_A,  \delta,\Delta) \ar@{->}[r]^-{\mathrm{Thm.}~\ref{thm:spfb}} & (\mathfrak{A}, [\ ,\ ]_\mathfrak{A}, \cdot_\mathfrak{A}, \delta_r, \Delta_r) 
		}
	\end{displaymath}
\end{corollary}

\begin{proposition}
	Let $(A, \cdot_{A}, \Delta_{r}, \Phi, \Psi)$ be a communicative and cocommutative factorizable differential ASI bialgebra, which corresponds to a symmetric Rota-Baxter differential Frobenius algebra $((A, \cdot_{A}, \Phi), \mathfrak{B}, P)$ of weight  $\lambda \neq 0$.
	Suppose that Eq.~\eqref{eq:vip} holds.
	Let $(A, [\ ,\ ]_{A}, \cdot_{A})$ be the induced Poisson algebra of $(A, \cdot_{A}, \Phi)$.
	Then $((A, [\ ,\ ]_{A}, \cdot_{A}), \mathfrak{B}, P)$ is a quadratic Rota-Baxter Poisson algebra of weight $\lambda$, to which the induced factorizable Poisson bialgebra $(A, [\ ,\ ]_{A}, \cdot_{A}, \delta_{r}, \Delta_{r})$ in Theorem~\ref{thm:dasi2pb} corresponds.
	In conclusion, we have the following commutative diagram
	\begin{displaymath}
		\xymatrix@C=2cm{ 
			(A, \cdot_A, \Delta_r, \Phi, \Psi) \ar@{->}[r]^{\mathrm{Thm.}~\ref{thm:dasi2pb}}\ar@{->}[d]^{\mathrm{Thm.}~\ref{thm:fdb2qrp}}&  (A, [\ ,\ ]_A, \cdot_A,  \delta_{r},\Delta_{r}) \ar@{->}[d]^{\mathrm{Thm.}~\ref{thm:fpb2qrp}} \\
			((A, \cdot_{A}, \Phi), \mathfrak{B}, P) \ar@<1.5ex>[u]^{\mathrm{Thm.}~\ref{thm:fdb2qrpx}} \ar@{->}[r] &((A, [\ ,\ ]_A, \cdot_A), \mathfrak{B}, P) \ar@<1.5ex>[u]^{\mathrm{Thm.}~\ref{thm:fpb2qrpx}}}
	\end{displaymath}
\end{proposition}
\begin{proof}
	For all $a, b, c \in A$, we have 
	\begin{align*}
		\mathfrak{B}([a, b]_{A}, c) &= \mathfrak{B}(\partial_{1}(a) \cdot_{A} \partial_{2}(b) - \partial_{2}(a) \cdot_{A} \partial_{1}(b), c) = \mathfrak{B}(a, \eth_{1}(\partial_{2}(b) \cdot_{A} c) - \eth_{2}(\partial_{1}(b) \cdot_{A} c)) \\
		&=\mathfrak{B}(a, \eth_{1}(\partial_{2}(b)) \cdot_{A} c - \partial_{2}(b) \cdot_{A} \partial_{1}(c) - \eth_{2}(\partial_{1}(b)) \cdot_{A} c + \partial_{1}(b) \cdot_{A} \partial_{2}(c) ) \\
		&= \mathfrak{B}(a, [b, c]_{A})
	\end{align*} 
	and
	\begin{align*}
		[P(a), P(b)]_A &= \partial_{1}(P(a)) \cdot_{A} \partial_{2}(P(b)) - \partial_{2}(P(a)) \cdot_{A} \partial_{1}(P(b)) \\
		&= P(\partial_{1}(a)) \cdot_{A} P (\partial_{2}(b)) - P(\partial_{2}(a)) \cdot_{A} P (\partial_{1}(b)) \\
		&= P(P(\partial_{1}(a)) \cdot_{A} \partial_{2}(b) + \partial_{1}(a) \cdot_{A} P(\partial_{2}(b)) + \lambda \partial_{1}(a) \cdot_{A} \partial_{2}(b)) \\
		&\quad - P(P(\partial_{2}(a)) \cdot_{A} \partial_{1}(b) + \partial_{2}(a) \cdot_{A} P(\partial_{1}(b)) + \lambda \partial_{2}(a) \cdot_{A} \partial_{1}(b)) \\
		&= P(\partial_{1}(P(a)) \cdot_{A} \partial_{2}(b) + \partial_{1}(a) \cdot_{A} \partial_{2}(P(b)) + \lambda \partial_{1}(a) \cdot_{A} \partial_{2}(b)) \\
		&\quad - P(\partial_{2}(P(a)) \cdot_{A} \partial_{1}(b) + \partial_{2}(a) \cdot_{A} \partial_{1}(P(b)) + \lambda \partial_{2}(a) \cdot_{A} \partial_{1}(b)) \\
		&= P([P(a), b]_A + [a, P(b)]_A + \lambda [a, b]_A).
	\end{align*}
	Therefore, $((A, [\ ,\ ]_{A}, \cdot_{A}), \mathfrak{B}, P)$ is a quadratic Rota-Baxter Poisson algebra of weight $\lambda$.
	Thus we have shown that the proposition is true.
\end{proof}

%%%%%%%%%%%%%%%%%%%%%%%%%%%%%%%%%%%%%%%%%%%%%%%%%%%%%%%%%%%%%%%%%%%%%%%%%%%%%%%%
%%%%%%%%%%%%%%%%%%%%%%%%%%%%%%%%%%%%%%%%%%%%%%%%%%%%%%%%%%%%%%%%%%%%%%%%%%%%%%%%
%%%%%%%%%%%%%%%%%%%%%%%%%%%%%%%%%%%%%%%%%%%%%%%%%%%%%%%%%%%%%%%%%%%%%%%%%%%%%%%%

%%%%%%%%%%%%%%%%%%%%%%%%%%%%%%%%%%%%%%%%%%%%%%%%%%%%%%%%%%%%%%%%%%%%%%%%%%%%%%%%
%%%%%%%%%%%%%%%%%%%%%%%%%%%%%%%%%%%%%%%%%%%%%%%%%%%%%%%%%%%%%%%%%%%%%%%%%%%%%%%%
%%%%%%%%%%%%%%%%%%%%%%%%%%%%%%%%%%%%%%%%%%%%%%%%%%%%%%%%%%%%%%%%%%%%%%%%%%%%%%%%
%\bigskip

\smallskip

\noindent
{\bf Data availability. } No new data were created or analyzed in this study.

%%%%%%%%%%%%%%%%%%%%%%%%%%%%%%%%%%%%%%%%%%%%%%%%%%%%%%%%%%%%%%%%%%%%%%%%%%%%%%%%

% \bibliographystyle{cim2022} % latex makebst
% \bibliography{ref.bib}

\end{document}